\documentclass[10pt,a4paper]{article}
\usepackage{latexsym}
\usepackage{amsthm,amsmath,amssymb}
\usepackage{fullpage}
\usepackage{float}
\usepackage{amsfonts}
\usepackage{graphicx}
\usepackage{amsmath}
\usepackage{color}
\usepackage{tikz}
\usepackage[colorlinks=true]{hyperref}

\newfloat{figure}{H}{lof}
\newfloat{table}{H}{lot}
\floatname{figure}{\figurename}
\floatname{table}{\tablename}

\newtheorem{theorem}{Theorem}[section]

\newtheorem{lemma}[theorem]{Lemma}

\newtheorem{proposition}[theorem]{Proposition}
\newtheorem{remark}[theorem]{Remark}

\numberwithin{equation}{section}

\newcommand{\R}{\mathbb{R}}
\newcommand{\N}{\mathbb{N}}

\begin{document}

\title{\textbf{On measure-valued solutions for a structured population model with transfers}}
	\author{\textsc{ Pierre Magal$^{{\rm (a)}}$ and   Ga\"el Raoul$^{{\rm (b)}}$}\\
		{\small \textit{$^{{\rm (a)}}$University of Bordeaux, IMB, UMR CNRS 5251, 33076 Bordeaux, France.}} \\
	{\small \textit{$^{{\rm (b)}}$CMAP, CNRS, Ecole polytechnique, Institut Polytechnique de Paris,
91128 Palaiseau, France.
	}} 
}

\maketitle

\begin{abstract} 

We consider a transfer operator where two interacting cells carrying non-negative traits transfer a random fraction of their trait to each other. These transfers can lead to population having  singular distributions in trait. We extend the definition of the transfer operator to non-negative measures with a finite second moment, and we discuss the regularity of the fixed distributions of that transfer operator. Finally, we consider a dynamic transfer model where an initial population distribution is affected by a transfer operator: we prove the existence and uniqueness of mild measure-valued solutions for that Cauchy problem.

\end{abstract}

\date{}


\medskip


\section{Introduction}
\label{sec:intro}

We consider a population structured by a trait $x\in\R_+$ that is the quantity of a specific good carried by an individual. We are interested in situations where transfers happen. Specifically, when two agents proceed to an exchange, where each individual sends a fraction of its trait to its partner. This type of model has been extensively studied as an econometrics model \cite{Chakraborti,Bisi2}, but they also play an important role in biological contexts. Transfer models are used to represent the effect of horizontal gene transfers \cite{Novozhilov,Billiard} in particular through the exchange of plasmids. They also appear in cancer cell models, where membrane P-glycoproteins (involved in resistance to chemotherapy) can be exchanged between cancer cells \cite{Hinow,Dyson,Pasquier,Pasquier2}. Our goal in this study is to develop a framework for a class of transfer operators acting on measure spaces, and to build an existence setting for measure-valued solutions for a Cauchy problem.

\medskip

In many kinetic models, the existence of solutions can be proven in adapted functional spaces. For instance, in the case of the Boltzmann equation, a priori estimates on the entropy can be used to apply a Dunford-Pettis compactness argument in $L^1$ \cite {Villani-Kinetic}, opening the way to an existence setting in that functional space. In structured population models with sexual reproductions, solutions are bounded thanks to a uniform bound on the reproduction kernel, and solutions can then be constructed in $L^\infty$. Existence and uniqueness setting in $L^1$ have also been proposed for transfer models: in \cite{MW00,M02}, a transfer operator combined to a  mutation operator ensured that solutions are bounded for positive times, and $L^1$ solutions were constructed. In \cite{Hinow}, a  specific transfer model was introduced, where pairs of interacting individuals exchange a fixed fraction of their traits. The authors show that if the initial condition belongs to $L^1(\R_+)$, so does the solution for any finite time. In the case of transfer models, however, limits of solutions belonging to functional spaces quickly appear: in \cite{Hinow}, it was shown that for a simple transfer model, solutions converge to a Dirac mass centered in the mean trait of the initial population. For some transfer kernels, singular distributions in trait can actually emerge for any positive time, even if the initial population is regular. It is then important to develop a well-grounded setting for transfer operators acting on measures, and for measure-valued solution to time-dependent transfer models.

\medskip

To construct measure-valued solutions of structured population or kinetic models, an efficient strategy is to introduce a metric on the set of Radon measures. The most popular metric is the bounded Lipschitz distance, that can be seen as a generalization of the $W_1-$Wasserstein distance to general measures. Equipped with this metric, the set of Radon measures with bounded first moment is complete \cite{Villani}, and regularization arguments can be used to construct solutions. This idea originated in kinetic theory, where it was used to construct solutions of Vlasov's equations \cite{Dobrushin,Neunzert, Sphohn}, and it was then used to construct solutions  for evolutionary biology models \cite{Canizo, Dull,Dull2,Ackleh} as well as swarming models \cite{Rosado,Ha}. An alternative idea is to use Fourier-based metrics, which are also related to Wasserstein distances, as discussed in \cite{Carrillo-rev}. The time-dependent transfer equations can indeed be reformulated (through a Fourier transform) into a simpler integro-differential equation, that can then be used to prove the existence of measure-valued solutions for Maxwellian Boltzmann equations \cite{Pulvirenti,Bobylev1,Bobylev2}. This strategy was developed for some transfer models in the context of econometrics in \cite{Mattes}. In \cite{Griette}, a semi-group approach was proposed to construct solutions measure-valued solutions to a class of transfer models, relying on the completeness of the set of measures equipped with the total variation distance. Finally, in \cite{Tanaka}, H. Tanaka constructed measure-valued solutions of a Maxwellian Boltzmann equation thanks to a Poisson point process. Another great idea in \cite{Tanaka} was to use Wasserstein distances to describe the long-time dynamics of solutions, an idea that was developed for transfer models in econometrics using Fourier-based metrics in e.g. \cite{Mattes,Bassetti2011,Bisi2009}, and through Wasserstein distance in \cite{Magal-Raoul}.

\medskip

In Section~\ref{sec:model}, we define the transfer model we consider in this manuscript. Then, in Section~\ref{sec:defoperator}, we generalize the transfer operator defined in Section~\ref{sec:model} into an operator on non-negative measures with finite second moments moments. We then investigate when the fixed distributions of the transfer operator contain a Dirac mass in Section~\ref{sec:fixed}. Finally, in Section~\ref{sec:time-problem} we construct solutions for a Cauchy problem for the transfer operator, using a tightness and identification strategy. The precise framework we develop provides a solid base for the development and analysis of transfer model for cancer cell populations (see Section~\ref{sec:model}). 

\medskip

Throughout this manuscript, we denote by $\int_0^{+\infty}f(x)\,d\mu(x)$ the integral of the Borel-measurable function $f$ against the Borel measure $\mu$ over $[0,+\infty)$, with $\{0\}$ included in the interval.  If $f\in L^1(\R_+)$, we denote by $f(x)\,dx$ the measure on $\R_+$ absolutely continuous with respect to the Lebesgue measure, with density $f$. we denote by $\mathcal P(X)$ the set of Radon probability measures on a topological space $X$, typically $X=[0,1]$ or $X=\R_+$. $\mathcal P_2(\R_+)$ designates the set of Borelian probability measures over $\R_+$ with finite second moment.
 
\section{The model}\label{sec:model}

\label{subsec:modelling-exchanges}

We consider the transfer model introduced in \cite{Magal-Raoul}, which can be seen as a generalization of the model introduced in \cite{Hinow,Dyson,Pasquier,Pasquier2}, and that describes the transfers of carrying P-glycoproteins in cancer cell populations. We refer in particular to \cite{Pasquier} for experimental results showing the effect of transfers on MCF-7 cancer cells, and for a comparison between the model introduced in \cite{Hinow} and these experimental data. When cells connect through tunneling nanotubes, the proteins are transferred between the two cells along these tunneling nanotubes (see \cite{Pasquier2}). More precisely, we refer to the cell as cell $1$ and cell $2$. If we denote by $x_1\geq 0$ (respectively $x_2\geq 0$) the quantity of protein carried by the cell $1$  (respectively cell $2$) before transfer (i.e. for ancestor cells), and  $x'_1\geq 0$ (respectively $x'_2\geq 0$) the quantity of protein carried by the cell $1$  (respectively cell $2$) after transfer. In the following we call the calls before transfer the  ancestor cells, and the cells after transfer the descendant cells. 
Then we can write 
\begin{equation}\label{2.1}
	\left\{ 
\begin{array}{l}
x_1'=x_1-X_1+X_2, \vspace{0.2cm}\\
x_2'=x_2-X_2+X_1,
\end{array}	
\right. 
\end{equation}
where $X_1 \in [0,x_1]$ is the quantity of protein given by the cell $1$, and  $X_2 \in [0, x_2]$ is the quantity of protein given by the cell $2$. 

\medskip 

In \cite{Hinow}, the fraction of protein sent by each cell was $p \in [0,1]$, a fixed number. Then $X_1=p \, x_1$ and $X_2=  p \, x_2$, and thus
\begin{equation}\label{2.2}
	\left\{ 
	\begin{array}{l}
		x_1'=x_1+ p \,(x_2-x_1) \vspace{0.2cm}\\
		x_2'=x_2+ p \, (x_1-x_2)
	\end{array}	
	\right. 
	\Leftrightarrow 
		\left\{ 
		\begin{array}{l}
			x_1'=(1-p) \, x_1+ p\, x_2 \vspace{0.2cm}\\
			x_2'=(1-p) \, x_2+ p \,x_1 .
		\end{array}	
		\right. 
\end{equation}
In this manuscript, we assume that the fraction of proteins sent is an independent random variable with law $B\in \mathcal P([0,1])$. If $x_1>0$, the law of $X_1$ is then $x_1B\left(\frac {dx}{x_1}\right)$, and we assume in this section that $B$ is absolutely continuous with respect to the Lebesgue measure, with a density that we denote by $\tilde B\in L^1([0,1])$ and that we extend into $\tilde B\in L^1(\R_+)$ thanks to $\tilde B(x)=0$ if $x>1$. The random variable $X_1$ has then a density $x\mapsto x_1\tilde B\left(\frac x{x_1}\right)$ with respect to Lebesgue's measure, and thanks to \eqref{2.1}, we can relate the probability law of $x_1'$ to the laws of $X_1$ and $X_2$. We deduce that the conditional probability that $x_1'\leq x$ (for $x\geq 0$) knowing $x_1>0$  and $x_2>0$ is given by 
\begin{align} 
\mathbb P\left(x_1'\leq x|x_1,\,x_2\right)&=\int_0^{+\infty}\left(\int_0^{x-x_1+w}\,\frac 1{x_2} \tilde B\left(\frac z{x_2}\right)dz\right)\,\frac 1{x_1} \tilde B\left(\frac w{x_1}\right)\,dw.  \label{2.6}
\end{align}
If we derive the right hand side of the above formula with respect to $x$, we obtain the function $K_{\tilde B}[x_1,x_2]\in L^1(\R_+)$ that we define as follows:
\begin{equation} \label{2.7}
K_{\tilde B}[x_1,x_2](x):=\frac 1{x_1x_2}\int_0^{+\infty} \tilde B\left(\frac{x-x_1+w}{x_2}\right)\tilde B\left(\frac w{x_1}\right)\,dw.
\end{equation}
When $x_1>0$ and $x_2>0$, $K_{\tilde B}[x_1,x_2]$ is a well defined function in $L^1(\R_+)$, since \eqref{2.6} can be seen as a convolution. Indeed, if $ \tilde B_1(w):= \tilde B\left(\frac{-w}{x_1}\right)$, $\tilde B_2(w):=\left(\frac{w}{x_2}\right)$, then $\tilde B_1,\,\tilde B_2\in L^1(\R_+)$ and the above formula becomes a convolution in $L^1(\R_+)$:
\begin{equation}\label{eq:KBconvolution}
	K_{\tilde B}[x_1,x_2](x)=\frac 1{x_1x_2}\int_0^{+\infty} \tilde B_1\big((x_1-x)-w\big)\tilde B_2\left(w\right)\,dw.
\end{equation}
We can define the bilinear operator $(u,v)\mapsto T_{\tilde B}[u,v]\in L^1(\R_+)$ where $\tilde u,\tilde v\in C_c(\R_+)$ as follows:
\begin{align}\label{def:T}
&T_{\tilde B}[\tilde u,\tilde v](x):=\int_0^{+\infty}\int_0^{+\infty}  K_{\tilde B}[x_1,x_2](x)\,\tilde u(x_1)\tilde v(x_2)\,dx_1\,dx_2,
\end{align}
that describe the law of the number of proteins carried by a cell uniformly chosen in the structured population of cells described by $\tilde u$, after it engages in a transfer with a cell uniformly selected in the structured population of cells described by $\tilde v$.

\section{Generalization of the transfer operator}\label{sec:defoperator}

In this section, we extend the definition of the measure $K_{\tilde B}[x_1,x_2]\,dx$ to cases where $B=\tilde B(z)\,dz$ is replaced by a general probability measure. We then proceed to extend the transfer operator $T_{\tilde B}[\tilde u,\tilde v]$, for $\tilde u,\tilde v\in L^1(\R_+)$, to situations where $B=\tilde B(z)\,dz$ is replaced by a general probability measure and $\tilde u(x)\,dx$, $\tilde v(x)\,dx$ are replaced by non-negative measures with finite second moments.

\subsection{Definition of the transfer operator when the kernal $B$ is a probability measure}

In this section, we extend the definition of the transfer operator $\mathbb T_B$ to cases where $B$ is a probability measure on $[0,1]$, and we will also extend the domain of the transfer operator from $C_c(\R_+)$ to Borel measures on $\R_+$. 

\medskip

Let $\mathbb M(\mathbb R_+)$ the set of non-negative Borel measures over $\R_+$. We introduce the following sets of non-negative Borel measures over $\R_+$ with finite  second moment:
\begin{equation}\label{def:MM2}
\mathbb M_2(\mathbb R_+)=\left\{u\in\mathbb M(\mathbb R_+);\,u\geq 0\textrm{ and }\int_0^{+\infty}(1+x^2)\,du(x)<+\infty\right\}.
\end{equation}
We can extend the definition of the operator $K_B$ (see \eqref{2.7}) as follows:

\begin{proposition}\label{Prop:barK}
Let $(B,x_1,x_2)\mapsto  K_{B}[x_1,x_2]$ the operator acting on $(L^1([0,1])\cap \mathcal P([0,1]))\times \R_+\times\R_+$ and valued in $\mathcal M_1(\mathbb R_+)$ defined by \eqref{2.7}. If we consider the weak topology of measures on $\mathcal P([0,1])$, there is a unique continuous  operator 
\begin{align}\label{def:barK}
\mathbb K:\mathcal P([0,1])\times \R_+\times\R_+&\to \mathbb M(\mathbb R_+),\\
(B,x_1,x_2)&\mapsto \mathbb K_{B}[x_1,x_2]\nonumber
\end{align}
that is equal to $(B,x_1,x_2)\mapsto  K_{\tilde B}[x_1,x_2](x)\,dx$ on the set of probability measure $B\in\mathcal P([0,1])$ that have a density $\tilde B\in L^1(\R_+)$ with respect to Lebesgue's measure. 

\smallskip

Moreover, for any $(B,x_1,x_2)\in \mathcal P([0,1])\times \R_+\times\R_+$, the distribution $\mathbb K_{B}[x_1,x_2]$ is a probability measure that satisfies
\begin{equation}\label{def:barK2}
\int_0^\infty  \varphi(x)\mathbb K_B[x_1,x_2](x)\,dx=\int_0^1\left(\int_0^1 \varphi(x_1(1-z_1)+x_2z_2)\,dB(z_1)\right)\,dB(z_2),
\end{equation}
for any $\varphi\in C_c(\R_+)$, and the support of $\mathbb K_B[x_1,x_2]$ is included in $[0,x_1+x_2]$.
\end{proposition}

\begin{proof}[Proof of Proposition~\ref{Prop:barK}]
Let $x_1,x_2\in\R_+$ and $B\in \mathcal P([0,1])$. We may define a function $\lambda_{B,x_1,x_2}$ acting on $C_c(\R_+)$ and valued in $\R$ as follows:
\begin{equation}\label{eq:lambdaK}
\lambda_{B,x_1,x_2}(\varphi):=\int_{[0,1]^2} \varphi(x_1(1-z_1)+x_2z_2)\,d(B\otimes B)(z_1,z_2).
\end{equation}
This function is linear in $\varphi$ and continuous: if the sequence $(\varphi_n)\in (C_c(\mathbb R_+))^{\mathbb N}$ converges to $\varphi_\infty\in C_c(\mathbb R_+)$ for $\|\cdot\|_\infty$, then
\begin{align*}
&\left|\lambda_{B,x_1,x_2}(\varphi_n)-\lambda_{B,x_1,x_2}(\varphi_\infty)\right|\leq  \|\varphi_n-\varphi_\infty\|_\infty\int_0^1\left(\int_0^1dB(z_1)\right)\,dB(z_2)\leq  \|\varphi_n-\varphi_\infty\|_\infty.
\end{align*}
We can then apply the Riesz–Markov–Kakutani representation Theorem to show that there is a unique measure $\mathbb K_B[x_1,x_2]\in \mathbb M(\mathbb R_+)$ such that 
\[\lambda_{B,x_1,x_2}(\varphi)=\int_0^{+\infty}\varphi(x)\,d(\mathbb K_B[x_1,x_2])(x),\]
for any $\varphi\in C_c(\R_+)$. The equality \eqref{eq:lambdaK} implies that $\int_0^{+\infty}\varphi(x)\,d(\mathbb K_B[x_1,x_2])(x)$ is non-negative as soon as $\varphi\in C_c(\R_+)$ is non-negative, which shows that $\mathbb K_B[x_1,x_2]$ is a non-negative measure thanks to Proposition~\ref{prop:appendix} in the Appendix. If $\varphi\in C_c(\R_+)$ a non-negative function such that $\varphi(x)=1$ if $x\in[0,x_1+x_2]$, then
\begin{align*}
\int_0^{+\infty}\varphi(x)\,d\left(\mathbb K_B[x_1,x_2]\right)(x)&=\int_{[0,1]^2} \varphi_l(x_1(1-z_1)+x_2z_2)\,d(B\otimes B)(z_1,z_2)\\
&=\int_{[0,1]^2} \,d(B\otimes B)(z_1,z_2)=1,
\end{align*}
which shows that $\mathbb K_B[x_1,x_2](\R_+)=1$, which, combined to its non-negativity, shows that $\mathbb K_B[x_1,x_2]$ is a probability measure. Moreover, if $\psi\in C_c(\mathbb R_+)$ is equal to $0$ on $[0,x_1+x_2]$, then the equality above implies $\int_0^{+\infty}\psi(x)\,d\left(\mathbb K_B[x_1,x_2]\right)(x)=0$, and since $\mathbb K_B[x_1,x_2]$ is non-negative, this equality implies $\mathbb K_B[x_1,x_2](A)=0$ if $A\subset [0,x_1+x_2]^c$. The support of $\mathbb K_B[x_1,x_2]$ is then included in $[0,x_1+x_2]$.

\medskip

To show that the operator $(B,x_1,x_2)\mapsto \mathbb K_B[x_1,x_2]$ is a continuous operator from $\mathcal P([0,1])\times\R_+\times\R_+$ to $\mathcal P(\R_+)$, let $(B^n,x_1^n,x_2^n)_{n\in\mathbb N}$ a sequence in $\mathcal P([0,1])\times\R_+\times\R_+$ such that $B_n$ converges to $B^\infty$ for the weak topology of measures, $x_1^n\to x_1^\infty$ in $\mathbb R_+$, and $x_2^n\to x_2^\infty$ in $\mathbb R_+$. Then, for $\varphi\in C_c(\R_+)$,
\begin{align}
&\int_0^{+\infty}\varphi(x)\,d(\mathbb K_{B^n}[x_1^n,x_2^n])(x)-\int_0^{+\infty}\varphi(x)\,d(\mathbb K_{B^\infty}[x_1^\infty,x_2^\infty])(x)\nonumber\\
&\quad =\iint_{[0,1]^2} \varphi(x_1^n(1-z_1)+x_2^nz_2)\,d(B^n\otimes B^n)(z_1,z_2)\nonumber\\
&\qquad -\iint_{[0,1]^2} \varphi(x_1^\infty(1-z_1)+x_2^\infty z_2)\,d(B^\infty\otimes B^\infty)(z_1,z_2)\nonumber\\
&\quad =
\iint_{[0,1]^2} \left[\varphi(x_1^n(1-z_1)+x_2^nz_2)-\varphi(x_1^\infty(1-z_1)+x_2^\infty z_2)\right]\,d(B^n\otimes B^n)(z_1,z_2)\nonumber\\
&\qquad + \iint_{[0,1]^2} \varphi(x_1^\infty(1-z_1)+x_2^\infty z_2)\,d(B^n\otimes B^n-B^\infty\otimes B^\infty)(z_1,z_2).\label{eq:cvKB}
\end{align}
Since $\varphi\in C_c(\R_+)$ is uniformly continuous on $[0,x_1+x_2+1]$, 
\[\left\|(z_1,z_2)\mapsto \varphi(x_1^n(1-z_1)+x_2^nz_2)-\varphi(x_1^\infty(1-z_1)+x_2^\infty z_2)\right\|_{L^\infty([0,1]^2)} \underset{n\to\infty}{\longrightarrow}0,\]
which proves that the first term on the right hand side of \eqref{eq:cvKB} converges to $0$ as $n\to\infty$. Moreover, the convergence of $(B^n)_{n\in\mathbb N}$ to $B^\infty$ for the weak topology of measures over $\R_+$ implies the convergence of $(B^n\otimes B^n)_{n\in\mathbb N}$ to $B^\infty\otimes B^\infty$ for the weak topology of measures over $\R_+^2$ (see Theorem~2.8 in \cite{Billingsley}), and the last term in \eqref{eq:cvKB} therefore converges to $0$ as $n\to\infty$. We have then proven
\[\int_0^{+\infty}\varphi(x)\,d(\mathbb K_{B^n}[x_1^n,x_2^n])(x)-\int_0^{+\infty}\varphi(x)\,d(\mathbb K_{B^\infty}[x_1^\infty,x_2^\infty])(x)\xrightarrow[n\to\infty]{}0,\]
which proves the continuity of the operator \eqref{def:barK}.  If  $B\in \mathcal P([0,1])$ is absolutely continuous with respect to Lebesgue's measure, with a density $\tilde B\in L^1([0,1])$, then, thanks to \eqref{eq:lambdaK} and a change of variable, for $x_1,x_2>0$,
\begin{align*}
\int_0^{+\infty} \varphi(x)\,d\big(\mathbb K_B[x_1,x_2]\big)(x)&=\lambda_{B,x_1,x_2}(\varphi)=\int_0^1\left(\int_0^1\varphi(x_1(1-z_1)+x_2z_2)\,dB(z_1)\right)\,dB(z_2)\\
&=\int_0^1\int_0^1\varphi(x_1(1-z_1)+x_2z_2)\tilde B(z_1)\tilde B(z_2)\,dz_1\,dz_2\\
&=\frac 1{x_1x_2}\int_0^{+\infty}\int_0^{+\infty}\varphi(x)\tilde B\left(\frac{x-x_1+w}{x_2}\right)\tilde B\left(\frac{w}{x_1}\right)\,dw\,dx\\
&=\int_0^{+\infty} \varphi(x)K_{\tilde B}[x_1,x_2](x)\,dx,
\end{align*}
thanks to \eqref{2.7}. This calculation shows that $K_{\tilde B}[x_1,x_2](x)\,dx=\mathbb K_B[x_1,x_2]$. If we denote by $ \mathcal P([0,1])\cap L^1([0,1])$ the set of probability measures absolutely continuous with respect to Lebesgue's measure, we notice that $\left(\mathcal P([0,1])\cap L^1([0,1])\right)\times \mathbb R_+\times\mathbb R_+$ density is a dense subset of $\mathcal P([0,1])\times \mathbb R_+\times\mathbb R_+$, where $\mathcal P([0,1])$ is endowed with the weak topology of measures. The application $(B,x_1,x_2)\mapsto \mathbb K_B[x_1,x_2]$ is therefore the only continuous extension of $(B,x_1,x_2)\mapsto K_{\tilde B}[x_1,x_2]$ to $\mathcal P([0,1])\times \mathbb R_+\times\mathbb R_+$.
\end{proof}
We extend the definition of the operator $ T_{\tilde B}[\tilde u,\tilde v]$ in the next proposition:
\begin{proposition}\label{Prop:barT}
Let $(\tilde B,\tilde u,\tilde v)\mapsto T_{\tilde B}[\tilde u,\tilde v]$ the operator acting on $ \{\tilde B\in L^1([0,1]);\,\tilde B\,dz\in \mathcal P([0,1]))\}\times C_c(\mathbb R_+)\times C_c(\mathbb R_+)$ and valued in $ C_c(\mathbb R_+)$ defined by \eqref{def:T}. There is a unique continuous operator 
\begin{align}
 {\mathbb T}:\mathcal P([0,1])\times \mathbb M_2(\mathbb R_+)\times \mathbb M_2(\mathbb R_+)&\to \mathbb M(\mathbb R_+),\label{def:barT}\\
(B,u,v)\quad &\mapsto {\mathbb T}_B[u,v],\nonumber
\end{align}
such that ${\mathbb T}_B[u,v]=T_{\tilde B}[\tilde u,\tilde v](x)\,dx$ when the probability measure $B\in\mathcal P([0,1])$ has a density $\tilde B\in L^1([0,1])$ with respect to Lebesgue's measure, and $u,v\in \mathbb M_2(\R_+)$ have densities $\tilde u,\tilde v\in C_c(\R_+)$ with respect to Lebesgue's measure. ${\mathbb T}_B[u,v]$ is a non-negative measure that satisfies:
\begin{align}
&\int_0^{+\infty} \varphi(x)d({\mathbb T}_B[u,v])(x)\nonumber\\
&\quad =\int_0^{+\infty}\int_0^{+\infty} \left(\int_0^1\left(\int_0^1 \varphi(x_1(1-z_1)+x_2z_2)\,dB(z_1)\right)\,dB(z_2)\right)\,du(x_1)\,dv(x_2),\label{def:barT2}
\end{align}
for any $\varphi\in C_c(\R_+)$, as well as
\begin{align}
\int_0^{+\infty} \varphi(x)d({\mathbb T}_B[u,v])(x)%
&=\int_0^{+\infty}\int_0^{+\infty} \left(\int_0^{+\infty}\varphi(x)\,d\left({\mathbb K}_B[x_1,x_2]\right)(x)\right)\,du(x_1)\,dv(x_2).\label{def:barT4}
\end{align}

\end{proposition}
We prove this result below, thanks to arguments similar to the ones developed in the proof of Proposition~\ref{Prop:barK}.

\begin{proof}[Proof of Proposition~\ref{Prop:barT}] 
Let $u,v\in C_c(\mathbb R_+)$ and $B\in \mathcal P([0,1])$. We define a function $\tilde \lambda_{B,u,v}$ acting on $C_c(\mathbb R_+)$ and valued in $\mathbb R$ as follows:
\begin{align}\label{eq:tildelambda}
&\tilde \lambda_{B,u,v}(\varphi): =\int_{[0,1]^2\times (\R_+)^2} \varphi\left(x_1 \, \left(1-z_1 \right)+x_2 z_2  \right)\, d\,B\otimes B\otimes u\otimes v( z_1,z_2,x_1,x_2).
\end{align}
This application is linear, and we check that it is continuous: If the sequence $(\varphi_n)\in (C_c(\mathbb R_+))^{\mathbb N}$ converges to $\varphi_\infty\in C_c(\mathbb R_+)$ for $\|\cdot\|_\infty$, then
\begin{align*}
&\left|\tilde\lambda_{B,u,v}(\varphi_n)-\tilde \lambda_{B,u,v}(\varphi_\infty)\right|\leq  \|\varphi_n-\varphi_\infty\|_{L^\infty(\R_+)}\int_{[0,1]^2\times (\R_+)^2}\, d\,B\otimes B\otimes u\otimes v( z_1,z_2,x_1,x_2)\underset{n\to\infty}{\longrightarrow}  0.
\end{align*}
We can then apply the Riesz–Markov–Kakutani representation  Theorem to show that there exists a unique measure ${\mathbb T}_B[u,v]\in \mathbb M(\mathbb R_+)$ such that $\lambda_{B,u,v}(\varphi)=\int_0^{+\infty}\varphi(x)\,d({\mathbb T}_B[u,v])(x)$. If $u,v\geq 0$, the measure ${\mathbb T}_B[u,v]$ is non-negative since $\lambda_{B,u,v}(\varphi)=\int_0^{+\infty}\varphi(x)\,d({\mathbb T}_B[u,v])(x)\geq 0$ for any non-negative test function $\varphi\in C_c(\R_+,\R_+)$, which allows us to apply  Proposition~\ref{prop:appendix} in the Appendix. To show that the operator $(B,u,v)\mapsto\mathbb T_B[u,v]$ is continuous, let $(B_n,u_n,v_n)_{n\in \N}$ a sequence in $\mathcal P([0,1])\times \mathbb M_2(\mathbb R_+)\times \mathbb M_2(\mathbb R_+)$ such that $(B_n)_{n\in\N}$ (resp. $(u_n)_{n\in\N}$, $(v_n)_{n\in\N}$) converges to $B_\infty\in\mathcal P([0,1])$ (resp. $u_\infty,v_\infty\in\mathbb M_2(\R_+)$) for the weak topology of measures. Thanks to  Theorem~2.8 in \cite{Billingsley}, the product measure $(B_n\otimes B_n\otimes u_n\otimes v_n)_{n\in\N}$ converges to $B_\infty\otimes B_\infty\otimes u_\infty\otimes v_\infty$ for the weak topology of measures, and then, for any $\varphi\in C_c(\mathbb R)$,
\begin{align}
&\left|\int_0^{+\infty} \varphi(x)\,d({\mathbb T}_{B_\infty}[u_\infty,v_\infty])(x)-\int_0^{+\infty} \varphi(x)\,d({\mathbb T}_{B_n}[u_n,v_n])(x)\right|=\left|\lambda_{B_\infty,u_\infty,v_\infty}(\varphi)-\lambda_{B_n,u_n,v_n}(\varphi)\right|\nonumber
\\
&\quad \leq \bigg|\int_{(\R_+)^2\times[0,1]^2} \varphi\left(x_1 \, \left(1-z_1 \right)+x_2 z_2  \right)\, d\left(B_\infty\otimes B_\infty\otimes u_\infty\otimes v_\infty\right)( z_1,z_2,x_1,x_2)\nonumber\\
&\qquad -\int_{(\R_+)^2\times[0,1]^2} \varphi\left(x_1 \, \left(1-z_1 \right)+x_2 z_2  \right)\, d\left(B_n\otimes B_n\otimes u_n\otimes v_n\right)( z_1,z_2,x_1,x_2)
\bigg|\underset{n\to\infty}{\longrightarrow}  0.\label{eq:decoup}
\end{align}
Since this convergence holds for any function $\varphi\in C_c(\mathbb R_+)$, we have proven the continuity of the operator $(B,u,v)\in \mathcal P([0,1])\times\mathbb M_2(\mathbb R_+)\times\mathbb M_2(\mathbb R_+)\mapsto {\mathbb T}_B[u,v]\in \mathbb M(\mathbb R_+)$. \eqref{eq:tildelambda} implies \eqref{def:barT2} thanks to a change of variable, while \eqref{def:barT4} follows from Proposition~\ref{Prop:barK}. If $(B,u,v)\in \mathcal P([0,1])\times \mathbb M_2(\R_+)\times \mathbb M_2(\R_+)$, with $B,\,u,\, v$ having densities  $\tilde B\in L^1([0,1])$, $\tilde u\in \times C_c(\mathbb R_+)$, $\tilde v\in C_c(\mathbb R_+)$ with respect to Lebesgue's measure, then \eqref{def:barT4} implies, for any $\varphi\in C_c(\R_+)$,
\begin{align*}
&\int_0^{+\infty} \varphi(x)d({\mathbb T}_B[u,v])(x)%
=\int_0^{+\infty}\int_0^{+\infty} \left(\int_0^{+\infty}\varphi(x){\mathbb K}_B[x_1,x_2](x)\,dx\right)\,du(x_1)\,dv(x_2)\\
&\quad =\int_0^{+\infty}\int_0^{+\infty} \left(\int_0^{+\infty}\varphi(x){K}_{\tilde B}[x_1,x_2](x)\,dx\right)\tilde u(x_1)\tilde v(x_2)\,dx_1\,dx_2=\int_0^{+\infty}\varphi(x){T}_{\tilde B}[\tilde u,\tilde v](x)\,dx,
\end{align*}
thanks to Proposition~\ref{Prop:barK} and the definition \eqref{def:T} of ${T}_{\tilde B}[\tilde u,\tilde v]$. Then ${\mathbb T}_B[u,v]={T}_{\tilde B}[\tilde u,\tilde v](x)\,dx$ for $(B,u,v)\in \mathcal P([0,1])\times \mathbb M_2(\R_+)\times \mathbb M_2(\R_+)$, with $B,\,u,\, v$ having densities  $\tilde B\in L^1([0,1])$, $\tilde u\in \times C_c(\mathbb R_+)$, $\tilde v\in C_c(\mathbb R_+)$ with respect to Lebesgue's measure. Since that set is a dense subset of $\mathcal P([0,1])\times\mathbb M_2(\mathbb R_+)\times\mathbb M_2(\mathbb R_+)$, the operator ${\mathbb T}$ is the unique continuous extension of $T$  to $\mathcal P([0,1])\times\mathbb M_2(\mathbb R_+)\times\mathbb M_2(\mathbb R_+)$.
\end{proof}

\subsection{Properties of the transfer operator $\mathbb T_B$}

We consider the moments of $\mathbb T_B[u,u]$, for $B\in\mathcal P([0,1])$ and $u\in \mathbb M(\R_+)$ in the proposition below. These moments calculations are essential to derive macroscopic models from mesoscopic models including such transfer operators, see \cite{Magal-Raoul}.

\begin{proposition}\label{Prop:barT2}
Let $B\in\mathcal P([0,1])$ and ${\mathbb  T}_B$ the transfer operator defined by \eqref{def:barT}. Then, for $u,v\in \mathbb M_2(\mathbb R_+)$,
\begin{equation}\label{est:moments0}
\int_0^{+\infty}\,d({\mathbb  T}_B[u,v])(x)=\left(\int_0^{+\infty}\,du(x)\right)\int_0^{+\infty}\,dv(x),
\end{equation}
\begin{equation}\label{est:moments1}
\int_0^{+\infty}x\,d({\mathbb  T}_B[u,v])(x)=\left(1-\lambda_{1,B}\right)\left(\int_0^{+\infty}\,dv(x)\right)\int_0^{+\infty} x\,du(x)+\lambda_{1,B}\left(\int_0^1\,du(x)\right)\int_0^{+\infty} x\,dv(x),
\end{equation}
\begin{align}
&\int_0^{+\infty}x^2\,d({\mathbb  T}_B[u,v])(x)=\left(1-2\lambda_{1,B}+\lambda_{2,B}\right)\left(\int_0^{+\infty}\,dv(x)\right)\int_0^{+\infty} x^2\,du(x)\nonumber\\
&\quad +\lambda_{2,B}\left(\int_0^{+\infty}\,du(x)\right)\int_0^{+\infty} x^2\,dv(x)+2\left(1-\lambda_{1,B}\right)\lambda_{1,B}\left(\int_0^{+\infty} x\,du(x)\right)\int_0^{+\infty} x\,dv(x),\label{est:moments2}
\end{align}
where 
\begin{equation}\label{def:lambdai}
\lambda_{1,B}=\int_0^1 x\,dB(x),\quad \lambda_{2,B}=\int_0^1 x^2\,dB(x).
\end{equation}
In particular, if $u,v\in \mathbb M_2(\mathbb R_+)$, then ${\mathbb  T}_B[u,v]\in \mathbb M_2(\mathbb R_+)$.
\end{proposition}
\begin{remark}\label{rem:momentsKB}
For $x_1,x_2\geq 0$, equality \eqref{def:barT4} implies
\[\mathbb T_B[\delta_{x_1},\delta_{x_2}]=\mathbb K_B[x_1,x_2],\]
where $\mathbb K_B[x_1,x_2]$ is a probability measure thanks to Proposition~\ref{Prop:barK}. Proposition~\ref{Prop:barT2} then implies the following properties on $\mathbb K_B[x_1,x_2]$:
\[ \int_0^{+\infty}x\,d\left(\mathbb K_B[x_1,x_2]\right)(x)=(1-\lambda_{1,B})x_1+\lambda_{1,B}x_2,\]
\[\int_0^{+\infty}x^2\,d\left(\mathbb K_B[x_1,x_2]\right)(x)=(1-2\lambda_{1,B}+\lambda_{2,B})x_1^2+\lambda_{2,B}x_2^2+2(1-\lambda_{1,B})\lambda_{1,B}x_1x_2.\]
Note also that
\[\int_0^{+\infty}x\,d\left(\mathbb T_B[u,u]\right)(x)=\int_0^{+\infty}x\,du(x),\quad \int_0^{+\infty} x\,d\left(\frac{\mathbb K_B[x_1,x_2]+\mathbb K_B[x_2,x_1]}2\right)(x)=\frac {x_1+x_2}2,\]
which reflects the fact that the transfers conserve the total number of proteins, as described by \eqref{2.1}.
\end{remark}

\begin{proof}[Proof of Proposition~\ref{Prop:barT2}]
\medskip

\noindent\textbf{Step 1: Case where $u,v\in\mathbb M_2(\R_+)$ are compactly supported}

Let $R>0$ and $u,v\in\mathbb M_2(\R_+)$ such that  $u(x)=v(x)=0$ for $x\geq R$. 

\medskip

Let $\phi\in C_c(\R_+)$ satisfying $\phi(x)=0$ for $x\leq 2R$. Then $u(x)=v(x)=0$ for $x\geq R$, and thanks to \eqref{def:barT2},
\begin{align*}
&\int_0^{+\infty} \phi(x){\mathbb T}_{B}[u,v](x)\,dx=\int_0^{+\infty}\int_0^{+\infty}\left(\int_0^{+\infty}\phi(x)\,d\left(\mathbb K_B[x_1,x_2]\right)(x)\right)\,du(x_1)\,dv(x_2) =0,
\end{align*}
since the support of $\mathbb K_B[x_1,x_2]$ is included in $[0,2R]$ for $x_1\in\textrm{supp }u$ and $x_2\in\textrm{supp }v$, thanks to Proposition~\ref{Prop:barK}.  The support of ${\mathbb T}_{B}[u,v]$ is then included in $[0,2R]$.

\medskip

Let $\varphi\in C_c(\mathbb R_+,\mathbb R_+)$ with $\varphi(x)=1$ for $x\leq 2R$. Then, for any $x\geq 0$ such that $\varphi(x)\neq 1$, we have $u(x)=v(x)= {\mathbb T}_{B}[u,v](x)=0$. For $\alpha,\beta,\gamma\geq 0$,
\begin{align*}
&\int_0^{+\infty}\left(\alpha+\beta x+\gamma x^2\right)\,d({\mathbb  T}_B[u,v])(x)=\int_0^{+\infty}\left(\alpha+\beta x+\gamma x^2\right)\varphi(x)\,d({\mathbb  T}_B[u,v])(x)\\
&\quad =\int_0^{+\infty}\int_0^{+\infty} \bigg(\int_0^1\bigg(\int_0^1 \left(\alpha+\beta (x_1(1-z_1)+x_2z_2)+\gamma (x_1(1-z_1)+x_2z_2)^2\right)\\
&\phantom{\quad =\int_0^{+\infty}\int_0^{+\infty} \bigg(\int_0^1\bigg(\int_0^1}\varphi(x_1(1-z_1)+x_2z_2)\,dB(z_1)\bigg)\,dB(z_2)\bigg)\,du(x_1)\,dv(x_2)\\
&\quad =\int_0^{+\infty}\int_0^{+\infty} \bigg(\int_0^1\bigg(\int_0^1 \left(\alpha+\beta (x_1(1-z_1)+x_2z_2)+\gamma (x_1(1-z_1)+x_2z_2)^2\right)\\
&\phantom{\quad =\int_0^{+\infty}\int_0^{+\infty} \bigg(\int_0^1\bigg(\int_0^1}\,dB(z_1)\bigg)\,dB(z_2)\bigg)\,du(x_1)\,dv(x_2),
\end{align*}
where we used Proposition~\ref{Prop:barT}. We compute :
\begin{align*}
&\int_0^1\bigg(\int_0^1 \left(\alpha+\beta (x_1(1-z_1)+x_2z_2)+\gamma (x_1(1-z_1)+x_2z_2)^2\right)\,dB(z_1)\bigg)\,dB(z_2)\\
&\quad =\alpha + \beta\left[x_1\left(1-\int_0^1 zdB(z)\right)+x_2\int_0^1 zdB(z)\right]+\gamma\bigg[x_1^2\left(1-2\int_0^1z\,dB(z)+\int_0^1 z^2\,dB(z)\right)\\
&\qquad  +x_2^2\int_0^1z^2\,dB(z)+2x_1x_2\left(1-\int_0^1z\,dB(z)\right)\int_0^1z\,dB(z)\bigg]\\
&\quad =\alpha + \beta\left[x_1\left(1-\lambda_{1,B}\right)+x_2\lambda_{1,B}\right]+ \gamma\left[x_1^2\left(1-2\lambda_{1,B}+\lambda_{2,B}\right)+x_2^2\lambda_{2,B}+2x_1x_2\left(1-\lambda_{1,B}\right)\lambda_{1,B}\right].
\end{align*}
Then,
\begin{align*}
&\int_0^{+\infty}\left(\alpha+\beta x+\gamma x^2\right)\,d({\mathbb  T}_B[u,v])(x)=\alpha \left(\int_0^{+\infty}\,du(x)\right)\int_0^{+\infty}\,dv(x)\\
&\quad + \beta\left[\left(1-\lambda_{1,B}\right)\left(\int_0^{+\infty}\,dv(x)\right)\int_0^{+\infty} x\,du(x)+\lambda_{1,B}\left(\int_0^{+\infty}\,du(x)\right)\int_0^{+\infty} x\,dv(x)\right]\\
&\quad +\gamma\bigg[\left(1-2\lambda_{1,B}+\lambda_{2,B}\right)\left(\int_0^{+\infty}\,dv(x)\right)\int_0^{+\infty} x^2\,du(x)+\lambda_{2,B}\left(\int_0^{+\infty}\,du(x)\right)\int_0^{+\infty} x^2\,dv(x)\\
&\phantom{\quad +\gamma\bigg[a}+2\left(1-\lambda_{1,B}\right)\lambda_{1,B}\left(\int_0^{+\infty} x\,du(x)\right)\int_0^{+\infty} x\,dv(x)\bigg].
\end{align*}

\medskip\textbf{Step 2: General case $u,v\in\mathbb M_2(\R_+)$}

For $u,v\in\mathbb M_2(\R_+)$, we notice that the applications $R\mapsto u(x) 1_{x\leq R}$, $R\mapsto v(x) 1_{x\leq R}$ and $R\mapsto {\mathbb T}_B[u1_{\,\cdot\,\leq R},v1_{\,\cdot\,\leq R}](x)$ are non-decreasing in $R$. We may then use the monotone convergence theorem to show, for $\alpha,\beta,\gamma\geq 0$,
\begin{align*}
&\int_0^{+\infty}\left(\alpha+\beta x+\gamma x^2\right)\,d({\mathbb  T}_B[u,v])(x)=\lim_{R\to\infty}\int_0^{+\infty}\left(\alpha+\beta x+\gamma x^2\right)\,d({\mathbb  T}_B[u1_{\,\cdot\,\leq R},v1_{\,\cdot\,\leq R}])(x)\\
&\quad =\lim_{R\to\infty}\bigg\{\alpha \left(\int_0^R\,du(x)\right)\int_0^R\,dv(x)\\
&\qquad + \beta\left[\left(1-\lambda_{1,B}\right)\left(\int_0^R\,dv(x)\right)\int_0^{R} x\,du(x)+\lambda_{1,B}\left(\int_0^R\,du(x)\right)\int_0^{R} x\,dv(x)\right]\\
&\qquad +\gamma\bigg[\left(1-2\lambda_{1,B}+\lambda_{2,B}\right)\left(\int_0^R\,dv(x)\right)\int_0^{R} x^2\,du(x)+\lambda_{2,B}\left(\int_0^R\,du(x)\right)\int_0^{R} x^2\,dv(x)\\
&\phantom{\quad +\gamma\bigg[a}+2\left(1-\lambda_{1,B}\right)\lambda_{1,B}\left(\int_0^{R} x\,du(x)\right)\int_0^{R} x\,dv(x)\bigg]\bigg\}\\
&\quad =\alpha \left(\int_0^{+\infty}\,du(x)\right)\int_0^{+\infty}\,dv(x)\\
&\qquad + \beta\left[\left(1-\lambda_{1,B}\right)\left(\int_0^{+\infty}\,dv(x)\right)\int_0^{+\infty} x\,du(x)+\lambda_{1,B}\left(\int_0^{+\infty}\,du(x)\right)\int_0^{+\infty} x\,dv(x)\right]\\
&\qquad +\gamma\bigg[\left(1-2\lambda_{1,B}+\lambda_{2,B}\right)\left(\int_0^{+\infty}\,dv(x)\right)\int_0^{+\infty} x^2\,du(x)+\lambda_{2,B}\left(\int_0^{+\infty}\,du(x)\right)\int_0^{+\infty} x^2\,dv(x)\\
&\phantom{\quad +\gamma\bigg[a}+2\left(1-\lambda_{1,B}\right)\lambda_{1,B}\left(\int_0^{+\infty} x\,du(x)\right)\int_0^{+\infty} x\,dv(x)\bigg].
\end{align*}

\end{proof}

\section{Properties of fixed distributions for the transfer operator}
\label{sec:fixed}

In this section, we focus our attention on fixed distributions for $u\mapsto\mathbb T_B[u,u]$, $u\not\equiv 0$,  that satisfy
\begin{equation}\label{4.1}
 \overline{u}=	 {\mathbb T}_B[ \overline{u},\overline{u}],
\end{equation}
as an equality between measures. An integration along $\R_+$ implies $\int_0^{+\infty} \,d\bar u(x)=\left(\int_0^{+\infty} \,d\bar u(x)\right)^2$, so that we can restrict our attention to probability measures, that is $\bar u\in\mathcal P_2(\R_+)$.

\subsection{Concentration at $x=0$}
\label{subsec:x0}

We notice that for any kernel $B\in \mathcal P(\mathbb R_+)$, the Dirac mass at the origin, ie $\bar u:=\delta_0$, is a fixed distribution for the transfer operator. Indeed, if $B\in \mathcal P(\mathbb R_+)$ and  $\varphi\in C_c(\mathbb R_+)$, then,
\begin{align*}
&\int_0^{+\infty} \varphi(x){\mathbb T}_B[\delta_0,\delta_0](x)\,dx\nonumber\\
&\quad =\int_0^{+\infty}\int_0^{+\infty} \left(\int_0^1\left(\int_0^1 \varphi(x_1(1-z_1)+x_2z_2)\,dB(z_1)\right)\,dB(z_2)\right)\,d(\delta_0)(x_1)\,d(\delta_0)(x_2)=\varphi(0),
\end{align*}
which implies ${\mathbb T}_B[\delta_0,\delta_0]=\delta_0$. Are there other fixed distributions for $u\mapsto \mathbb T_B[u,u]$ that contain a Dirac mass at $x=0$? We show in the next proposition that such fixed distributions exist when the transfer kernel $B\in \mathcal P(\mathbb R_+)$ contains Dirac masses in both $x=0$ and $x=1$:
\begin{proposition}\label{prop:delta01}
If $B\in\mathcal P([0,1])$ satisfies $B\geq \alpha(\delta_0+\delta_1)$ for some $\alpha>0$ and $\bar u\in \mathcal P_2(\mathbb R_+)$ satisfies ${\mathbb T}_B[\bar u,\bar u]=\bar u$, then $\bar u\geq \alpha^2 \delta_0$.
\end{proposition}

\begin{proof}[Proof of Proposition~\ref{prop:delta01}] For $\varphi\in C_c(\mathbb R_+)$,
\begin{align*}
&\int_0^{+\infty} \varphi(x){\mathbb T}_B[\bar u,\bar u](x)\,dx\nonumber\\
&\quad =\int_0^{+\infty}\int_0^{+\infty} \left(\int_0^1\left(\int_0^1 \varphi(x_1(1-z_1)+x_2z_2)\,dB(z_1)\right)\,dB(z_2)\right)\,d\bar u(x_1)\,d\bar u(x_2)\\
&\quad \geq \alpha^2\int_0^{+\infty}\int_0^{+\infty} \left(\int_0^1\left(\int_0^1 \varphi(x_1(1-z_1)+x_2z_2)\,d(\delta_1)(z_1)\right)\,d(\delta_0)(z_2)\right)\,d\bar u(x_1)\,d\bar u(x_2)\\
&\quad = \alpha^2\int_0^{+\infty}\int_0^{+\infty}\varphi(0)\,d\bar u(x_1)\,d\bar u(x_2)=\alpha^2\varphi(0),
\end{align*}
and then $\bar u=\bar{\mathbb T}_B[\bar u,\bar u]$ implies $\bar u\geq \alpha^2\delta_0$.
\end{proof}
Surprisingly, we show that as soon as $B$ does not contain a Dirac mass at both extremities of the interval $[0,1]$ (and $B\neq \delta_0$, $B\neq \delta_1$), the only fixed distribution for the transfer operator containing mass at $x=0$ is $\delta_0$:

\begin{proposition}\label{prop:delta02}
Let $B\in\mathcal P([0,1])$, such that $B\neq \delta_0$ and $B\neq \delta_1$. We assume that $B$ either satisfies:
\[B(\{0\})=0\quad \textrm{ or }\quad B(\{1\})=0.\]
 If $\bar u\in \mathcal P_2(\mathbb R_+^*)$ satisfies ${\mathbb T}_B[\bar u,\bar u]=\bar u$, then either $\bar u=\delta_0$ or $\bar u(\{0\})=0$. 
\end{proposition}
\begin{remark}
Notice that if $B=\delta_0$ or $B=\delta_1$, then the transfer operator reduces to the identity, as we can check from \eqref{def:barT2}: if $B=\delta_0$,
\begin{align*}
&\int_0^{+\infty} \varphi(x)\,d({\mathbb T}_B[u,u])(x)\nonumber\\
&\quad =\int_0^{+\infty}\int_0^{+\infty} \left(\int_0^1\left(\int_0^1 \varphi(x_1(1-z_1)+x_2z_2)\,d(\delta_0)(z_1)\right)\,d(\delta_0)(z_2)\right)\,du(x_1)\,du(x_2)\\
&\quad =\int_0^{+\infty}\varphi(x_1)\,du(x_1),
\end{align*}
and if $B=\delta_1$, a similar calculation can be made. Therefore, when $B=\delta_0$ or $B=\delta_1$, any measure $\bar u\in\mathcal P(\mathbb R_+)\cap \mathbb M_2(\R_+)$ is a fixed points of $u\mapsto\mathbb T_B[u,u]$, and in particular measures containing a Dirac mass at the origin are fixed points of $u\mapsto\mathbb T_B[u,u]$.
\end{remark}

\begin{proof}[Proof of Proposition~\ref{prop:delta02}] We assume that $B(\{0\})=0$ and $\nu:=B(\{1\})\in [0,1)$.

\smallskip

We use a contradiction argument: we suppose that $\bar u(\{0\})=\eta\in (0,1)$.  
For any $\varepsilon>0$, we can define $\varphi^\varepsilon\in C_c(\mathbb R_+)$ such that $\varphi^\varepsilon(x)=0$ for $x\geq \varepsilon$, and $\|\varphi^\varepsilon\|_\infty\leq 1$. Then, for $x_1,x_2\in\R_+$,
\begin{align*}
&\varphi^\varepsilon(x_1(1-z_1)+x_2z_2)\leq 1_{x_1(1-z_1)\leq\varepsilon}1_{x_2z_2\leq \varepsilon}\leq \left(1_{x_1\leq \bar x}+1_{x_1\geq \bar x}1_{\bar x(1-z_1)\leq \varepsilon}\right)+\left(1_{x_2\leq \bar x}+1_{x_2\geq \bar x}1_{\bar xz_2\leq \varepsilon}\right)\\
&\quad \leq 1_{x_1\leq \bar x}1_{x_2\leq \bar x}+1_{x_1\leq \bar x}1_{x_2\geq \bar x}1_{\bar xz_2\leq \varepsilon}+1_{x_1\geq \bar x}1_{\bar x(1-z_1)\leq \varepsilon}1_{x_2\leq \bar x}+1_{x_1\geq \bar x}1_{\bar x(1-z_1)\leq \varepsilon}1_{x_2\geq \bar x}1_{\bar xz_2\leq \varepsilon}.
\end{align*}
We can regroup the second and last term on the right hand side of this equation (and the fact that $\|\varphi^\varepsilon\|_\infty\leq 1$) to obtain:
\begin{align*}
&\varphi^\varepsilon(x_1(1-z_1)+x_2z_2)\\
&\quad \leq 1_{x_1\leq \bar x}1_{x_2\leq \bar x}+1_{x_1\geq \bar x}1_{\bar x(1-z_1)\leq \varepsilon}1_{x_2\leq \bar x}+\left(1_{x_1\leq \bar x}1_{x_2\geq \bar x}1_{\bar xz_2\leq \varepsilon}+1_{x_1\geq \bar x}1_{\bar x(1-z_1)\leq \varepsilon}1_{x_2\geq \bar x}1_{\bar xz_2\leq \varepsilon}\right)\\
&\quad \leq 1_{x_1\leq \bar x}1_{x_2\leq \bar x}+1_{x_1\geq \bar x}1_{\bar x(1-z_1)\leq \varepsilon}1_{x_2\leq \bar x}+\left(1_{x_1\leq \bar x}+1_{x_1\geq \bar x}1_{\bar x(1-z_1)\leq \varepsilon}\right)1_{x_2\geq \bar x}1_{\bar xz_2\leq \varepsilon}\\
&\quad \leq 1_{x_1\leq \bar x}1_{x_2\leq \bar x}+1_{x_1\geq \bar x}1_{\bar x(1-z_1)\leq \varepsilon}1_{x_2\leq \bar x}+1_{x_2\geq \bar x}1_{\bar xz_2\leq \varepsilon}.
\end{align*}
We use this inequality to estimate:
\begin{align}
&\int_0^{+\infty} \varphi^\varepsilon(x){\mathbb T}_B[\bar u,\bar u](x)\,dx\label{eq:Tuu}\\
&\quad =\int_0^{+\infty}\int_0^{+\infty} \left(\int_0^1\left(\int_0^1  \varphi(x_1(1-z_1)+x_2z_2)\,dB(z_1)\right)\,dB(z_2)\right)\,d\bar u(x_1)\,d\bar u(x_2)\nonumber\\
&\quad \leq \int_0^{\bar x}\int_0^{\bar x} \left(\int_0^1\left(\int_0^1  \,dB(z_1)\right)\,dB(z_2)\right)\,d\bar u(x_1)\,d\bar u(x_2)\nonumber\\
&\qquad +\int_0^{\bar x}\int_{\bar x}^{+\infty} \left(\int_0^{1}\left(\int_{1-\varepsilon/\bar x}^1  \,dB(z_1)\right)\,dB(z_2)\right)\,d\bar u(x_1)\,d\bar u(x_2)\nonumber\\ 
&\qquad + \int_{\bar x}^{+\infty}\int_{0}^{+\infty} \left(\int_0^{\varepsilon/\bar x}\left(\int_0^1  \,dB(z_1)\right)\,dB(z_2)\right)\,d\bar u(x_1)\,d\bar u(x_2)\nonumber\\
&\quad \leq \left(\int_0^{\bar x}\,d\bar u(x)\right)^2+\left(\int_0^{\bar x}\,d\bar u(x)\right)\left(\int_{\bar x}^{+\infty} \,du(x)\right)\left(\int_{1-\varepsilon/\bar x}^1\,dB(z)\right)+\left(\int_{\bar x}^{+\infty}\,d\bar u\right)\left(\int_0^{\varepsilon/\bar x}\,dB\right)\nonumber\\
&\quad \leq \left(\int_0^{\bar x}\,d\bar u(x)\right)\left[\left(\int_0^{\bar x}\,d\bar u(x)\right)+\left(\int_{\bar x}^{+\infty} \,du(x)\right)\left(\int_{1-\varepsilon/\bar x}^1\,dB(z)\right)-\left(\int_0^{\varepsilon/\bar x}\,dB(z)\right)\right]\nonumber\\
&\qquad +\left(\int_0^{\varepsilon/\bar x}\,dB(z)\right).\nonumber
\end{align}
Since $\bar u(\{0\}=\eta$ , $B(\{0\})=0$ and $B(\{1\})=\nu$, this inequality and the outer regularity of the Borel measures $B$ and $\bar u$ imply:
\begin{align}
&\int_0^{+\infty} \varphi^\varepsilon(x){\mathbb T}_B[\bar u,\bar u](x)\,dx\nonumber\\
&\quad \leq \left(\int_0^{\bar x}\,d\bar u(x)\right)\left[(\eta+o_{\bar x\to 0}(1))+(1-\eta)(\nu+o_{\varepsilon/\bar x\to 0}(1))-o_{\varepsilon/\bar x\to 0}(1)\right]+o_{\varepsilon/\bar x\to 0}(1),\nonumber
\end{align}
where the notation $o_{\bar x\to 0}(1)$ (resp. $o_{\varepsilon/\bar x\to 0}(1)$) designates a function that converges to $0$ when $\bar x\to 0$ (resp. $\varepsilon/\bar x\to 0$). Thanks to our contradiction assumption, $\int_0^{\bar x}\,d\bar u(x)\geq \eta>0$ for any $\bar x>0$, and then
\begin{align}
&\int_0^{+\infty} \varphi(x){\mathbb T}_B[\bar u,\bar u](x)\,dx\nonumber\\
&\quad \leq \left(\int_0^{\bar x}\,d\bar u(x)\right)\left[(\eta+o_{\bar x\to 0}(1))+(1-\eta)(\nu+o_{\varepsilon/\bar x\to 0}(1))-o_{\varepsilon/\bar x\to 0}(1)+o_{\varepsilon/\bar x\to 0}(1)\right].\label{eq:contradi}
\end{align}
The assumption $\eta<1$ implies $\eta<\frac {1-\nu}{1-\nu}$, or, equivalently, $\eta-\eta\nu+\nu<1$. We can then select $\bar x>0$ small enough to have $(\eta+o_{\bar x\to 0}(1))-\eta\nu+\nu<1$. This strict inequality and \eqref{eq:contradi} imply, provided $\varepsilon>0$ is small enough,
\begin{align*}
\int_0^{+\infty} \varphi(x){\mathbb T}_B[\bar u,\bar u](x)\,dx&< \left(\int_0^{\bar x}\,d\bar u(x)\right),
\end{align*}
and then
\begin{align*}
\eta=\bar u(\{0\})&\leq \int_0^\varepsilon\,d\bar u=\int_0^\varepsilon\,d\left({\mathbb T}_B[\bar u,\bar  u]\right)(x)\leq \int_0^{+\infty}\varphi(x)\,d\left({\mathbb T}_B[\bar u,\bar  u]\right)(x)<\eta,
\end{align*}
This is absurd, which concludes our contradiction argument. We have then proven, under the assumptions on $B$ made in the proposition, that $\bar u(\{0\})=\eta\in\{0,1\}$, where $\eta=1$ corresponds to $\bar u=\delta_0$. A similar argument can be used when $\int_{[1-\varepsilon,1)}\,dB\xrightarrow[\varepsilon\to 0]{}0$, and $\lim_{z\to 0}\int_{[0,z]}\,dB=\nu<1$: it is sufficient to use a change of variable to replace $\varphi(x_1(1-z_1)+x_2z_2)$ by $\varphi(x_1z_1+x_2(1-z_2))$ in \eqref{eq:Tuu} and repeat the calculations above.

\end{proof}

\subsection{Concentrations away from $x=0$}
\label{subsec:propequi}

In the introduction, we mentioned the so-called \emph{deterministic transfer} case, which corresponds to a transfer kernel $B(x)=\delta_p(x)$ for some $p\in[0,1]$. If $B(x)=\delta_p(x)$ and $Z\in\mathbb R_+^*$, the equality \eqref{def:barT2} satisfied by $\mathbb T_B$ implies 
\begin{align*}
&\int_0^{+\infty}\varphi(x)\,d( {\mathbb T}_{\delta_p}[\delta_Z,\delta_Z])(x)\\
&\quad =\int_0^{+\infty}\int_0^{+\infty} \left(\int_0^1\left(\int_0^1 \varphi(x_1(1-z_1)+x_2z_2)\,d(\delta_p)(z_1)\right)\,d(\delta_p)(z_2)\right)\,d(\delta_Z)(x_1)\,d(\delta_Z)(x_2)=\varphi(Z),
\end{align*}
which shows that $ {\mathbb T}[\delta_Z,\delta_Z]=\delta_Z$, and the probability measure $\bar u=\delta_Z$ is  a fixed distribution for $u\mapsto{\mathbb T}_B[u,u]$. These concentrated fixed distributions were already noticed and studied in \cite{Hinow}. We show in the next proposition that the singular fixed distributions $\delta_Z$ are actually specific to cases where $B$ is a Dirac mass:

\begin{proposition}\label{prop:dirac}
Let $B\in \mathcal P([0,1])$ and $Z\in\mathbb R_+^*$. If ${\mathbb T}_B$ is defined as in \eqref{def:barT} and satisfies
\[{\mathbb T}_B[\delta_Z,\delta_Z]=\delta_Z,\]
then $B=\delta_p$ for some $p\in[0,1]$.
\end{proposition}

\begin{remark} 
If $B=\delta_p$ for some $p\in[0,1]$, then the transfer exchanges described by $u\mapsto\mathbb T_B[u,u]$ are given by \eqref{2.2}, and the exchanges bring the traits closer: if $x_1,x_2\geq 0$, then 
\[x_1',x_2'\in [\min(x_1,x_2),\max(x_1,x_2)].\]
With other transfer kernels $B$, it is possible to have $x_1',x_2'\notin [\min(x_1,x_2),\max(x_1,x_2)]$. We then expect less concentrated fixed distributions as described by Proposition~\ref{prop:dirac}.

\smallskip

If $B=\delta_0$, any measure $u\in\mathbb M_2(\R_+)$ is a fixed distribution for the transfer operator, since for $\varphi\in C_c(\R_+)$,
\begin{align*}
&\int_0^{+\infty}\varphi(x)\,d( {\mathbb T}_{\delta_p}[u,u])(x)\\
&\quad =\int_0^{+\infty}\int_0^{+\infty} \left(\int_0^1\left(\int_0^1 \varphi(x_1(1-z_1)+x_2z_2)\,d(\delta_0)(z_1)\right)\,d(\delta_0)(z_2)\right)\,du(x_1)\,du(x_2)\\
&\quad =\int_0^{+\infty}\varphi(x)\,du(x),
\end{align*}
and the same property holds if $B=\delta_1$. This result can be understood intuitively, since  $B=\delta_0$ corresponds to a case where interacting cells do not transfer any fraction of their trait, while if $B=\delta_1$, the traits of the two interacting cells are exchanged. Neither case affects the trait distribution of the population.
\end{remark}

\begin{proof}[Proof of Proposition~\ref{prop:dirac}]
Let $p:=\int_{[0,1]} x \,dB(x)$. Since $B$ is a probability distribution on $[0,1]$, $p=0$ implies $B=\delta_0$ and $p=1$ implies $B=\delta_1$. If $p\in (0,1)$ and  $B\neq \delta_p$, then $\int_{[0,p)\cup(p,1]}\,dB(z)>0$, which implies $\int_{[0,p)}|z-p|\,dB(z)+\int_{(p,1]}|z-p|\,dB(z)>0$, while
\begin{align*}
0&=\int_{[0,1]}z\,dB(z)-p= \int_{[0,1]}(z-p)\,dB(z)=\int_{[0,p)}(z-p)\,dB(z)+\int_{(p,1]}(z-p)\,dB(z)\\
&=-\int_{[0,p)}|z-p|\,dB(z)+\int_{(p,1]}|z-p|\,dB(z),
\end{align*}
which implies $\int_{[0,p)}|z-p|\,dB(z)>0$, and then $\int_{[0,p)}\,dB(z)>0$ as well as $\int_{(p,1]}\,dB(z)>0$. The inner regularity of the Borel measure $B$ then implies
 \[\int_{[0,p-\varepsilon]}\,dB(z)>0,\quad \int_{[p+\varepsilon,1]}\,dB(z)>0,\]
for some $\varepsilon>0$. Let $\varphi\in C_c(\mathbb R_+^*)$ satisfying $\varphi(Z)=0$, $\varphi\geq 0$ and $\inf_{x\in [(1+2\varepsilon)Z,2Z]}\varphi(x)>0$. Then
\begin{align*}
&\int_0^{+\infty}\varphi(x){\mathbb T}_B[\delta_Z,\delta_Z](x)\,dx\\
&=\int_0^{+\infty}\int_0^{+\infty}\int_0^1\int_0^1 \varphi(x_1(1-z_1)+x_2z_2)\,dB(z_1)\,dB(z_2)\,d(\delta_Z)(x_1)\,d(\delta_Z)(x_2)\\
&\quad  =\int_0^1\int_0^1 \varphi(Z(1-z_1)+Zz_2)\,dB(z_1)\,dB(z_2)\geq \int_{p+\varepsilon}^1\int_0^{p-\varepsilon}\varphi\big(Z(1-z_1+z_2)\big)\,dB(z_1)\,dB(z_2)\\
&\quad \geq \left(\inf_{[(1+2\varepsilon)Z,2Z]}\varphi\right)\left(\int_{0}^{p-\varepsilon}\,dB(z_1)\right)\left(\int_{p+\varepsilon}^1 \,dB(z_2)\right)>0,
\end{align*}
which shows that ${\mathbb T}_B[\delta_Z,\delta_Z]\neq \delta_Z$.
\end{proof}
If the transfer kernel $B$ is not a Dirac mass, the proposition above shows that no Dirac mass can be a fixed distribution for the transfer operator.  The next proposition shows under limited conditions on the transfer kernel $B$, any fixed distribution of the transfer operator different from $\delta_0$ is absolutely continuous with respect to the Lebesgue measure:

\begin{proposition}\label{prop:BL1}
Let $B\in \mathcal P([0,1])$ and ${\mathbb T}_B$ defined as in \eqref{def:barT}. We assume that $B\geq f_B\,dx$, where $f_B\in C^0([0,1],\mathbb R_+)$ satisfies $\int_0^1f_B(z)\,dz>0$, and we also assume that $B(\{0,1\})=0$. If $\bar u\in \mathcal P_2(\mathbb R_+^*)$ satisfies ${\mathbb T}_B[\bar u,\bar u]=\bar u$ and $\bar u\neq \delta_0$, then $\frac{\partial \bar u}{\partial x}\in L^1(\mathbb R_+)$ and $\frac{\partial \bar u}{\partial x}(x)>0$ for $x\in\R_+$.
\end{proposition}

\begin{remark}
If $B$ is continuous with respect to the Lebesgue measure, then $\mathbb T_B[u]$ is continuous with respect to the Lebesgue measure for any $u\in\mathcal P_2(\R_+)$, since $\mathbb T_B[u]$ can be seen as a convolution of $B$ and a measure. This property implies that fixed points are continuous with respect to the Lebesgue measure as soon as $B$ has that same property. The statement above provides a weaker assumption on $B$ ensuring that fixed distributions are continuous with respect to the Lebesgue measure.
\end{remark}

\begin{proof}[Proof of Proposition~\ref{prop:BL1}]

Since $f_B\in C^0([0,1],\mathbb R_+)$ satisfies $\int_0^1f_B(z)\,dz>0$, there is $p\in (0,1)$, $\varepsilon\in(0,1)$ and $b>0$ such that
\[[p-\varepsilon,p+\varepsilon]\subset (0,1),\quad B\geq b1_{[p-\varepsilon,p+\varepsilon]}.\]
Let $\bar x>0$ and $0<\alpha$ small. For $x_1,x_2\in[\bar x-\alpha,\bar x+\alpha]$ and $\varphi\in C_c(\R_+)$ such that $\varphi\geq 0$,
\begin{align*}
&\int_0^{+\infty}\varphi(x){ \mathbb K}_B[x_1,x_2](x)\,dx=\int_0^1\left(\int_0^1\varphi(x_1(1-z_1)+x_2z_2)\,dB(z_1)\right)\,dB(z_2)\\
&\quad \geq  b^2\int_{\R}\int_{\R}1_{[p-\varepsilon,p+\varepsilon]}(z_1)1_{[p-\varepsilon,p+\varepsilon]}(z_2)\varphi(x_1(1-z_1)+x_2z_2)\,dz_2\,dz_1\\
&\quad = \frac{b^2}{x_2}\int_{\R}\left(\int_{\R}1_{[p-\varepsilon,p+\varepsilon]}(z_1)1_{[p-\varepsilon,p+\varepsilon]}\left(\frac{Z-x_1(1-z_1)}{x_2}\right)\,dz_1\right)\varphi(Z)\,dZ\\
&\quad \geq  \frac{b^2}{\bar x+\alpha}\int_{\bar x-\alpha}^{\bar x+\alpha}\left(\int_{\R}1_{[p-\varepsilon,p+\varepsilon]}(z_1)1_{\left[1+\frac{x_2(p-\varepsilon)-Z}{x_1},1+\frac{x_2(p+\varepsilon)-Z}{x_1}\right]}\left(z_1\right)\,dz_1\right)\varphi(Z)\,dZ
\end{align*}
We notice that if $Z\in[\bar x-\alpha,\bar x+\alpha]$, then 
\[1+\frac{x_2(p-\varepsilon)-Z}{x_1}\leq 1+\frac{(\bar x+\alpha)(p-\varepsilon)}{\bar x-\alpha}-\frac{\bar x-\alpha}{\bar x+\alpha}\leq p-\varepsilon+\frac{4\alpha}{\bar x-\alpha}\leq p-\frac \varepsilon 2,\]
if we select $\alpha:=\frac{\varepsilon\bar x}{16}$. Similarly, $1+\frac{x_2(p+\varepsilon)-Z}{x_1}\geq p+\frac \varepsilon 2$. Then,
\begin{align*}
&\int_0^{+\infty}\varphi(x) \mathbb K_B[x_1,x_2](x)\,dx\geq  \frac{\varepsilon b^2}{2\bar x+1}\int_{\bar x-\frac{\varepsilon\bar x}{16}}^{\bar x+\frac{\varepsilon\bar x}{16}}\varphi(Z)\,dZ.
\end{align*}
Then, for $\psi\in C_c(\R_+)$ such that $\psi\geq 0$, thanks to \eqref{def:barT4}, for any $\bar x>0$,
\begin{align*}
\int_0^{+\infty}\psi(x)\,d\left(\mathbb T_B[\bar u,\bar u]\right)(x)&=\int_0^{+\infty}\int_0^{+\infty} \left(\int_0^{+\infty} \psi(x)\,d\left(\mathbb K_B[x_1,x_2]\right)(x)\right)\,d\bar u(x_1)\,d\bar u(x_2)\\
&\geq \int_{\bar x-\varepsilon\bar x/16}^{\bar x+\varepsilon\bar x/16} \int_{\bar x-\varepsilon\bar x/16}^{\bar x+\varepsilon\bar x/16} \left(\int_{\bar x-\varepsilon\bar x/16}^{\bar x+\varepsilon\bar x/16} \psi(x)\,d\left(\mathbb K_B[x_1,x_2]\right)(x)\right)\,d\bar u(x_1)\,d\bar u(\bar x)\\
&\geq \frac{\varepsilon b^2}{2\bar x+1}\left(\int_{\bar x-\varepsilon\bar x/16}^{\bar x+\varepsilon\bar x/16}\psi(x)\,dx\right) \left( \int_{\bar x-\varepsilon\bar x/16}^{\bar x+\varepsilon\bar x/16} \,d\bar u(x)\right)^2,
\end{align*}
which implies
\begin{align}\label{eq:demi-harnack}
\mathbb T_B[\bar u,\bar u]&\geq \frac{\varepsilon b^2}{2\bar x+1} \left( \int_{\bar x-\varepsilon\bar x/16}^{\bar x+\varepsilon\bar x/16} \,d\bar u(x)\right)^21_{[\bar x-\varepsilon\bar x/16,\bar x+\varepsilon\bar x/16]}\,dx,
\end{align}
thanks to Lusin's Theorem (see Proposition~\ref{prop:appendix} for a related argument). Since $\bar u\in\mathcal P(\R_+)\cap\mathbb M_2(\R_+)$ is a fixed distribution for $u\mapsto \mathbb T_B[u,u]$, we have $\mathbb  T_B[\bar u,\bar u]=\bar u$, and then this inequality shows that provided $\bar u\neq \delta_0$, the support of $\bar u$ is open in $\mathbb R_+$. Since the support of a measure is closed by definition, it is equal to $\mathbb R_+$, that is $\bar u>0$ on $\mathbb R_+$: for any Borel set $A\subset \R_+$ of positive Lebesgue measure, $\bar u(A)>0$.

\medskip

Thanks to the Lebesgue-Radon-Nikodym theorem (\cite{Rudin}, Section 6.10 p. 121), $\bar u$ can be decomposed into $\bar u=\hat u+\tilde u$, $\hat u,\tilde u\geq 0$, with $\tilde u\in L^1(\mathbb R_+)$ while $\hat u\in\mathbb M(\mathbb R_+)$ and the support of $\hat u$ is of Lebesgue measure $0$:
\[\hat \Omega:=\textrm{supp }\hat u,\quad \lambda(\hat\Omega)=0,\]
where $\lambda$ stands for the Lebesgue measure on $\R_+$. Since $\bar u>0$, we have 
\begin{equation}\label{eq:intu1}
\int_0^{+\infty}\,d\hat u(x)<1,
\end{equation}
and we can decompose $\mathbb T_B[\bar u,\bar u]=\mathbb T_B[\bar u,\tilde  u]+\mathbb T_B[\tilde u,\hat u]+\mathbb T_B[\hat u,\hat u]$. Let $R>0$. Thanks to Lusin's theorem (see ), there is $\varphi_i\in C_c([0,R))$ with $\|\varphi_l\|_{L^\infty([0,R))}\leq 1$ such that $\varphi_l\to 1_{\hat \Omega\cap[0,R)}$ almost everywhere for $\bar u$, which implies $\varphi_l\to 0$ almost everywhere for the Lebesgue measure $\lambda$, and then for the measure the measure $\tilde u$ restricted to $[0,R)$. Then, thanks to the dominated convergence theorem and \eqref{def:barT2},
\begin{align*}
&{\mathbb T}_B[\tilde u,\hat u](\hat\Omega\cap [0,R))=\lim_{l\to\infty} \int_0^{+\infty} \varphi_l(x)d({\mathbb T}_B[\tilde u,\hat u])(x)\nonumber\\
&\quad =\lim_{l\to\infty} \int_{[0,1]\times\R_+\times[0,1]}\left(\int_0^{+\infty}\varphi_l(x_1(1-z_1)+x_2z_2)\,d\tilde u(x_1)\right)\,dB\otimes\hat u\otimes B(z_2,x_2,z_1)\\
&=\lim_{l\to\infty} \int_{[0,1]\times\R_+\times[0,1)}\left(\int_0^{+\infty}\varphi_l(x_1(1-z_1)+x_2z_2)\,d\tilde u(x_1)\right)\,dB\otimes\hat u\otimes B(z_2,x_2,z_1)\\
&\quad +B(\{1\})\left(\int_0^{+\infty}\,d\tilde u(x_1)\right)\lim_{l\to\infty} \int_{[0,1]\times\R_+}\varphi_l(x_2z_2)\,dB\otimes\hat u(z_2,x_2).
\end{align*}
Thanks to the dominated convergence theorem, $\lim_{l\to\infty}\int_0^{+\infty}\varphi_l(x_1(1-z_1)+x_2z_2)\,d\tilde u(x_1)=0$, since $\varphi_l(x_1(1-z_1)+x_2z_2)\xrightarrow[l\to\infty]{}0$ almost everywhere for the Lebesgue measure (notice that $(1-z_1)\neq 0$). Using the dominated convergence theorem once more shows
\[\lim_{l\to\infty} \int_{[0,1]\times\R_+\times[0,1)}\left(\int_0^{+\infty}\varphi_l(x_1(1-z_1)+x_2z_2)\,d\tilde u(x_1)\right)\,dB\otimes\hat u\otimes B(z_2,x_2,z_1)\xrightarrow[l\to\infty]{}0,\]
and thus
\begin{align*}
\left|{\mathbb T}_B[\tilde u,\hat u](\hat\Omega)\right|&=\lim_{R\to\infty}\left|{\mathbb T}_B[\tilde u,\hat u](\hat\Omega\cap[0,R])\right|\leq B(\{1\})\left(\int_0^{+\infty}\,d\tilde u(x)\right)\left(\int_0^{+\infty}\,d\hat u(x)\right)\\
&\leq B(\{1\})\left(1-\int_0^{+\infty}\,d\hat u(x)\right)\left(\int_0^{+\infty}\,d\hat u(x)\right).
\end{align*}
The same argument can be used to show
\begin{align*}
&{\mathbb T}_B[\bar u,\tilde u](\hat\Omega\cap [0,R))=B(\{0\})\left(\int_0^{+\infty}\,d\tilde u(x_2)\right)\lim_{l\to\infty} \int_{[0,1]\times\R_+}\varphi_l(x_1(1-z_1))\,dB\otimes\bar u(z_1,x_1)\\
&\quad =B(\{0\})\left(\int_0^{+\infty}\,d\tilde u(x_1)\right)\lim_{l\to\infty} \int_0^1\left(\int_0^{+\infty}\varphi_l(x_1(1-z_1))\,d\bar u(x_1)\right)\,dB(z_1)\\
&\quad =B(\{0\})\left(\int_0^{+\infty}\,d\tilde u(x_1)\right)\int_0^1\left(\int_0^{+\infty}1_{\hat \Omega\cap[0,R]}(x_1(1-z_1))\,d\hat u(x_1)\right)\,dB(z_1)\\
&\quad \leq B(\{0\})\left(\int_0^{+\infty}\,d\tilde u(x_1)\right)\left(\int_0^{+\infty}\,d\hat u(x)\right),
\end{align*}
where we have used the fact that $\int_0^{+\infty}1_{\hat \Omega\cap[0,R]}(x_1(1-z_1))\,d\tilde u(x_1)=0$, since $\hat \Omega\cap[0,R]$ is negligible with respect to the Lebesgue measure. These two estimates and  $\mathbb T_{B}[\hat u,\hat u](\R_+)=\hat u(\R_+)^2$, which is a consequence of  \eqref{est:moments0}, imply:
\begin{align*}
\hat u(\R_+)=u1_{\Omega}=\mathbb T_B[\bar u,\bar u](\Omega)\leq \left[B(\{0,1\})+\hat u(\R_+)\right]\hat u(\R_+).
\end{align*}
In particular, if $B(\{0,1\})=0$, this inequality implies $\hat u(\R_+)\in\{0,1\}$. Since we assume $\bar u\neq \delta_0$, we have $\int_{(0,+\infty)}\,du(x)>0$, and there is $\bar x>0$ such that $\int_{\bar x-\varepsilon \bar x/16}^{\bar x+\varepsilon \bar x/16}\,d\bar u(x)>0$. Since $\bar u=\mathbb T_B[\bar u,\bar u]$, estimate  \eqref{eq:demi-harnack} implies that $\tilde u$, the continuous part of $\bar u$, satisfies  
\[\forall x\in[\bar x-\varepsilon\bar x/16,\bar x+\varepsilon\bar x/16],\quad \tilde u(x)\geq \frac{\varepsilon b^2}{2\bar x+1} \left( \int_{\bar x-\varepsilon\bar x/16}^{\bar x+\varepsilon\bar x/16} \,d\bar u(x)\right)^2>0.\]
Then $\hat u(\R_+)=1-\tilde u(\R_+)<1$, and $\hat u(\R_+)\in\{0,1\}$ then implies $\hat u(\R_+)=0$, and $\tilde u=\frac{\partial \bar u}{\partial x}$.

\end{proof}

\section{The Cauchy problem for a dynamic transfer model}\label{sec:time-problem}

In this section, we are interested in the definition and existence of measure-valued solutions of the following time-dependent problem:
\begin{align}
\partial_t n_t(x)&=g_t(x)+h(t,x)n_t(x)+\frac 1{\int_0^{+\infty}\,dn_t(y)}\mathbb T_B[n_t,n_t](x),\label{eq:pbt}
\end{align}
where  $g\in L^\infty(\R_+,\mathbb M_2(\mathbb R_+))$ is non-negative and given, $h\in C^0(\R_+^2,\R)$ is given, and the initial condition is $n_0=n^{ini}\in\mathbb M_2(\R_+)$. The measure $g$ and the function $h$ are introduced to facilitate the use of the results of this manuscript in future works, typically to prove the existence of solutions for more complex models where transfers are combined to other biological phenomena. If $g\equiv 0$ and $h\equiv -1$, \eqref{eq:pbt} corresponds to the model of a population structured by the trait $x\in\R_+$ that is only affected by exchanges, similar to the model considered in \cite{Hinow}. The functions $g$ and $h$ in \eqref{eq:pbt} should allow the use of the existence results of this manuscript to study more complex models. We assume that $g$ and $h$ satisfy:
\begin{equation}\label{eq:condg}
\int_0^{+\infty} (1+ y^2)\,dg_t(y)\leq \bar C_A,\quad h(t,x)\leq \bar C_A,\quad  \|\partial_x h(t,\cdot)\|_{L^\infty(\R_+)}\leq \bar C_A,\quad h(t,x)\geq -\bar h_A(x),
\end{equation}
for a constant $\bar C_A>0$, $t,x\geq 0$ and $\bar h_A\in L^\infty_{loc}(\R_+)$. From here on, we denote by $\bar C>0$ constants that only depend on the initial condition $n^{ini}$ and $\bar C_A$, $\bar h_A$, appearing in Assumption \eqref{eq:condg}. In the next section, we will introduce a truncated problem, with a truncation (and regularization) parameter $\varepsilon>0$. In the proof of Lemma~\ref{lem:existencetruncated}, we denote by $C>0$ constants for a  given $\varepsilon>0$.

\subsection{Existence of solutions for a truncated problem}\label{subsec:truncated}

To define a truncated problem corresponding to \eqref{eq:pbt}, let $\Gamma_1\in C_c(\mathbb R_+,\mathbb R_+)$ with $\int_0^{+\infty}\Gamma_1(x)\,dx=1$ and $\textrm{supp }\Gamma_1\subset(1,2)$. We can then define the rescaled function $\Gamma_\varepsilon$ and the smoothed transfer kernel $\tilde {\mathbb K}_B^\varepsilon[x_1,x_2]$ as follows:
\begin{equation}\label{def:Kkeps}
\Gamma_\varepsilon(x):=\frac 1\varepsilon \Gamma_{1}\left(\frac x\varepsilon\right),\quad \tilde {\mathbb K}_B^\varepsilon[x_1,x_2](x):= \left(\Gamma_{\varepsilon}\ast{\mathbb K}_B[x_1,x_2]\right)(x)1_{ x\leq 1/\varepsilon},
\end{equation}
where the symbol $*$ stands for the convolution in $\mathbb R$, that is 
\[\left(\Gamma_\varepsilon\ast {\mathbb K}_B[x_1,x_2]\right)(x)= \int_0^{+\infty}\Gamma_\varepsilon(y-x) \,d\left({\mathbb K}_B[x_1,x_2]\right)(y).\]
For any given $\varepsilon>0$, $\|\tilde {\mathbb K}_B^\varepsilon[x_1,x_2]\|_{ L^\infty(\mathbb R_+)}\leq\|\Gamma_\varepsilon\|_{ L^\infty(\mathbb R_+)}$ is uniformly bounded in $x_1,\,x_2\in\R_+$. Moreover,  $\tilde {\mathbb K}_B^\varepsilon[x_1,x_2]$ is non-negative, and its support is included in $[\varepsilon,1/\varepsilon]$ for any $x_1,x_2\in\R_+$. The truncated problem we consider in this section is:
\begin{align}
n^\varepsilon(t,x)&=e^{\int_0^th(s,x)\,ds}\left(\Gamma_{\varepsilon}\ast n^{ini}\right)(x)1_{ x\leq 1/\varepsilon}+\int_0^t e^{\int_s^t h(\sigma,x)\,d\sigma}\left(\Gamma_{\varepsilon}\ast g_s\right)(x)1_{x\leq 1/\varepsilon}\,ds\nonumber\\
&\quad +\displaystyle\int_0^t\frac{e^{\int_s^t h(\sigma,x)\,d\sigma}}{\int_\varepsilon^{1/\varepsilon}n^\varepsilon(s,y)\,dy}\bigg(\iint_{\mathbb R_+^2} \tilde {\mathbb K}_B^\varepsilon[x_1,x_2](x)n^\varepsilon(s,x_1)n^\varepsilon(s,x_2)\,dx_2\,dx_1\bigg)\,ds.\label{eq:truncatedpb}
\end{align}
In the lemma below we prove the existence of solutions for this problem and provide some estimates that are uniform in $\varepsilon>0$ on these solutions:
\begin{lemma}\label{lem:existencetruncated}
Let $B\in \mathcal P([0,1])$, $\mathbb K_B$ defined as in \eqref{def:barT}, $g\in L^\infty(\R_+,\mathbb M_2(\R_+))$, $h\in C^0(\R_+^2,\R)$ satisfying \eqref{eq:condg}, and $n^{ini}\in\mathbb M_2(\R_+)$. There is $\bar C=\bar C(n^{ini},\bar C_A,\bar h_A)>0$ such that if $\varepsilon>0$ is small enough, there is a unique solution $n^\varepsilon\in L^\infty_{loc}(\R_+,L^\infty(\R_+,\R_+))$ to the truncated problem defined by \eqref{def:Kkeps}-\eqref{eq:truncatedpb}, that is non-negative and the support of $x\mapsto n^\varepsilon(t,x)$ is included in $[\varepsilon,1/\varepsilon]$ for $t\geq 0$. It satisfies, for $t\in\R_+$,
\begin{equation}\label{eq:ineqmom2}
\int_0^{+\infty} n^\varepsilon(t,x)\,dx\leq \bar C e^{\bar Ct},\quad \int_0^{+\infty} x^2n^\varepsilon(t,x)\,dx\leq \bar Ce^{\bar Ct}, \quad \int_0^{+\infty} n^\varepsilon(t,x)\,dx\geq \frac 1{\bar C}e^{-\bar Ct}.
\end{equation}
Moreover, $t\mapsto n^\varepsilon(t,\cdot)$ is locally Lipschitz continuous for the total variation topology, i.e., for $t_1,t_2\geq 0$,
\begin{equation}\label{eq:Lipeps}
d_{TV}\Big(n^\varepsilon(t_1,\cdot)\,dx,n^\varepsilon(t_2,\cdot)\,dx\Big)\leq \bar Ce^{\bar C\max(t_1,t_2)}|t_1-t_2|.
\end{equation}

\end{lemma}
In this statement, $d_{TV}$ stands for the total variation distance between two measures, that is, with $\nu,\mu\in \mathbb M(\R_+)$,
\begin{equation}\label{def:TV}
d_{TV}(\mu,\nu)=\sup_{\varphi\in L^\infty(\R_+),\,\|\varphi\|_{L^\infty(\R_+)}\leq 1} \int_0^{+\infty}\varphi(x)\,d(\mu-\nu)(x).
\end{equation}

\begin{proof}
We consider $\varepsilon>0$ as a fixed parameter in this proof. 

 \smallskip

\noindent\textbf{Step 1: Construction of a global solution of \eqref{eq:truncatedpb}}

We construct a sequence $(n_k^\varepsilon)_{k\in\mathbb N}$, for some $T>0$ that we will select later. For $k=0$ and $t\in[0,T]$, we define 
\[n_0^\varepsilon(t,x)=\left(\Gamma_{\varepsilon}\ast n^{ini}\right)(x)1_{ x\leq 1/\varepsilon},\]
so that $n_0^\varepsilon(t,\cdot)\in \mathbb M_2(\mathbb R_+)\cap L^\infty(\R_+)$. Then $\int_0^{+\infty}(1+x^2)n_0(t,x)\,dy$ and $\|n_0(t,\cdot)\|_{L^\infty(\R_+)}$ are bounded independently from $t\geq 0$, while the support of $n_0^\varepsilon(t,\cdot)$ is included in $[\varepsilon,1/\varepsilon]$ for $t\geq 0$. We construct $n_{k+1}^\varepsilon$ for $k\in\N$ by induction: given $n_k^\varepsilon\in L^\infty([0,T]\times\R_+)$ a non-negative function, we can define $\left((t,x)\mapsto n_{k+1}^\varepsilon(t,x)\right)\in L^\infty([0,T]\times\mathbb [\varepsilon,1/\varepsilon],\R_+)$ with the following formula, that is an ordinary differential equation in integral form:
\begin{align}
n_{k+1}^\varepsilon(t,x)&=e^{\int_0^th(s,x)\,ds}\left(\Gamma_{\varepsilon}\ast n^{ini}\right)(x)1_{ x\leq 1/\varepsilon}+\int_0^t e^{\int_s^t h(\sigma,x)\,d\sigma}\left(\Gamma_{\varepsilon}\ast g_s\right)(x)1_{x\leq 1/\varepsilon}\,ds\nonumber\\
&\quad +\int_0^t \frac{e^{\int_s^t h(\sigma,x)\,d\sigma}}{\int_0^{+\infty}n^\varepsilon_k(t,y)\,dy}\bigg(\iint_{\mathbb R_+^2} \tilde {\mathbb K}_B^\varepsilon[x_1,x_2](x)n_k^\varepsilon(s,x_1)n_k^\varepsilon(s,x_2)\,dx_1\,dx_2\bigg)\,ds.
\label{eq:scheme}
\end{align}
Since $0\leq \tilde {\mathbb K}_B^\varepsilon[x_1,x_2](x)\leq \|\Gamma_\varepsilon\|_{L^\infty(\R_+)}$ for $x,x_1,x_2\in\R_+$, the Cauchy-Lipschitz theorem shows that the function $n_{k+1}^\varepsilon\in L^\infty([0,T]\times\R_+)$ is a well defined continuous  function. Since $\left(\Gamma_{\varepsilon}\ast n^{ini}\right)\geq 0$, $\left(\Gamma_{\varepsilon}\ast g_s\right)\geq 0$, $\tilde K_B^\varepsilon[x_1,x_2]\geq 0$ and $n_k^\varepsilon\geq 0$, each term in \eqref{eq:scheme} is non-negative and then the function $n^\varepsilon_{k+1}$ itself is non-negative. Thanks to the definition \eqref{def:Kkeps} of $\Gamma_\varepsilon$ and $\tilde{\mathbb K}_B^\varepsilon[x_1,x_2](x)$, as well as the indicator functions multiplying the initial condition $\Gamma_\varepsilon\ast n^{ini}$ and source term $\Gamma_\varepsilon\ast g_s$, the support of $x\mapsto n_{k+1}^\varepsilon(t,x)$ is included in $[\varepsilon,1/\varepsilon]$ for $t\geq 0$. An integration along $x$ of \eqref{eq:scheme} shows
\begin{align}
&\int_0^{+\infty}n_{k+1}^\varepsilon(t,x)\,dx\leq e^{\int_0^th(s,x)\,ds}\int_0^\infty \,dn^{ini}(x)+\int_0^te^{\int_s^th(\sigma,x)\,d\sigma}\left(\int_0^{+\infty}\,dg_s(x)\right)\,ds\nonumber\\
&\quad +\int_0^t\frac{e^{\int_s^t h(\sigma,x)\,d\sigma}}{\int_0^{+\infty}n^\varepsilon(t,y)\,dy}\bigg(\iint_{\mathbb R_+^2}\bigg(\int_0^{+\infty}\tilde{\mathbb K}_B^\varepsilon[x_1,x_2](x)\,dx\bigg)n_k^\varepsilon(s,x_1)n_k^\varepsilon(s,x_2)\,dx_1\,dx_2 \bigg)\,ds\nonumber\\
&\quad \leq \bar C_1e^{\bar C_3t}+\bar C_2\int_0^t e^{\bar C_3(t-s)} \left(\int_0^{+\infty}n_{k}^\varepsilon(s,x)\,dx\right)\,ds,\label{est:L1}
\end{align}
where we have introduced numbered constants to improve readability (notice in particular that $\bar C_i=\bar C_i(n^{ini},\bar C_A,\bar h_A)>0$, $i\in\{1,2,3\}$ are independent from $\varepsilon>0$). Thanks to the definition of $n_0^{\varepsilon_k}$, we have $\int_0^{+\infty}n_{0}^\varepsilon(s,x)\,dx\leq \bar C_1e^{(\bar C_2+\bar C_3)t}$, and we use an induction to show a similar estimate on $n_k^\varepsilon$ for any $k\in\mathbb N$: if
\begin{equation}\label{eq:estL1}
\int_0^{+\infty}n_{k}^\varepsilon(s,x)\,dx\leq \bar C_1e^{(\bar C_2+\bar C_3)t}
\end{equation}
for $k\in\N$, then inequality \eqref{est:L1} proves that \eqref{eq:estL1} also holds for $k+1$. This induction proves the upper bound on $\int_0^{+\infty} n^\varepsilon(t,x)\,dx$ stated in \eqref{eq:ineqmom2}. Since $n^\varepsilon_k(t,x)=0$ for $x\geq 1/\varepsilon$, we have also  proven that  $n_k^\varepsilon\in L^\infty([0,T],\mathbb M_2(\mathbb R_+))$. Moreover, estimate \eqref{eq:estL1} implies
\begin{align}
&\|n_{k+1}^\varepsilon(t,\cdot)\|_{L^\infty(\R_+)}\leq\|\Gamma_{\varepsilon}\|_{L^\infty(\R_+)}\bigg\{e^{\int_0^th(s,x)\,ds}\int_0^\infty \,dn^{ini}(x)+\int_0^t e^{\int_s^th(\sigma,x)\,d\sigma}\int_0^{+\infty}\,dg_s(x)\,ds\nonumber\\
&\qquad +\displaystyle\int_0^t\frac{\max_{x\in\R_+}e^{\int_s^th(\sigma,x)\,d\sigma}}{\int_0^{+\infty}n_k^\varepsilon(s,y)\,dy}\bigg(\iint_{\mathbb R_+^2}n_k^\varepsilon(s,x_1)n_k^\varepsilon(s,x_2)\,dx_2\,dx_1\bigg)\,ds\bigg\}\leq Ce^{Ct},\label{est:Linf}
\end{align}
where $C>0$ is a constant that is not independent from $\varepsilon>0$. In particular, $n_{k+1}^\varepsilon(t,\cdot)\in L^\infty(\mathbb R_+)$ for $t\in[0,T]$, and we can compute, for $x\in\R_+$,
\begin{align}
&\|n_{k+2}^\varepsilon(t,\cdot)-n_{k+1}^\varepsilon(t,\cdot)\|_{L^\infty(\mathbb R_+)}\nonumber\\
&\quad\leq \max_{x\in\R_+}\bigg|\int_0^t\frac{e^{\int_s^th(\sigma,x)\,d\sigma}}{\int_0^{+\infty}n_{k+1}^\varepsilon(s,y)\,dy}\bigg(\iint_{\mathbb R_+^2}\tilde{\mathbb K}_B^\varepsilon[x_1,x_2](x)  n_{k+1}^\varepsilon(s,x_1)n_{k+1}^\varepsilon(s,x_2)\,dx_1\,dx_2\bigg)\,ds\nonumber\\
&\quad -\int_0^t\frac{e^{\int_s^th(\sigma,x)\,d\sigma}}{\int_0^{+\infty}n_k^\varepsilon(s,y)\,dy}\bigg(\iint_{\mathbb R_+^2}\tilde{\mathbb K}_B^\varepsilon[x_1,x_2](x)n_k^\varepsilon(s,x_1)n_k^\varepsilon(s,x_2)\,dx_1\,dx_2\bigg)\,ds\bigg|\nonumber\\
&\quad \leq \max_{x\in\R_+}\int_0^te^{\int_s^th(\sigma,x)\,d\sigma}\bigg(\iint_{\mathbb R_+^2}\tilde{\mathbb K}_B^\varepsilon[x_1,x_2](x)\left|(n_{k+1}^\varepsilon(s,x_1)-n_{k}^\varepsilon(s,x_1)\right|\frac{n_{k+1}^\varepsilon(s,x_2)}{\int_0^{+\infty}n_{k+1}^\varepsilon(s,y)\,dy}\,dx_1\,dx_2\bigg)\,ds\nonumber\\
&\qquad +\max_{x\in\R_+}\int_0^te^{\int_s^th(\sigma,x)\,d\sigma}\bigg(\iint_{\mathbb R_+^2}\tilde{\mathbb K}_B^\varepsilon[x_1,x_2](x)\left|\int_0^{+\infty}(n_{k}^\varepsilon(s,y)-n_{k+1}^\varepsilon(s,y))\,dy\right| \nonumber\\
&\qquad\phantom{+\int_0^te^{\int_s^th(\sigma,x)\,d\sigma}\bigg(\iint_{\mathbb R_+^2}}\frac{n_{k}^\varepsilon(s,x_1)}{\int_0^{+\infty}n_{k}^\varepsilon(s,y)\,dy}\frac{n_{k+1}^\varepsilon(s,x_2)}{\int_0^{+\infty}n_{k+1}^\varepsilon(s,y)\,dy}\,dx_1\,dx_2\bigg)\,ds\nonumber\\
&\qquad +\max_{x\in\R_+}\int_0^te^{\int_s^th(\sigma,x)\,d\sigma}\bigg(\iint_{\mathbb R_+^2}\tilde{\mathbb K}_B^\varepsilon[x_1,x_2](x)\frac{n_{k}^\varepsilon(s,x_1)}{\int_0^{+\infty}n_{k}^\varepsilon(s,y)\,dy}\left|n_{k+1}^\varepsilon(s,x_2)-n_{k}^\varepsilon(s,x_2)\right|\,dx_1\,dx_2\bigg)\,ds.\label{eq:contraction}
\end{align}
To estimate the first term on the right hand side of this equation, we notice that, for $x\in\R_+$ and $x_1,x_2\in[\varepsilon,1/\varepsilon]$,
\begin{align*}
&\int_0^{+\infty} \tilde {\mathbb K}_B^\varepsilon[x_1,x_2](x)\,dx_1=\int_\varepsilon^{1/\varepsilon}\int_0^{+\infty} \Gamma_\varepsilon(x-y) \tilde{\mathbb K}_B^\varepsilon[x_1,x_2](y)\,dy\,dx_1\\
&\quad \leq \|\Gamma_\varepsilon\|_{ L^\infty(\R_+)}\int_\varepsilon^{1/\varepsilon}\left(\int_0^{+\infty} \mathbb K_B[x_1,x_2](y)\,dy\right)\,dx_1\leq \frac 1\varepsilon\|\Gamma_\varepsilon\|_{ L^\infty(\R_+)},
\end{align*}
and then, since $n_{k+1}^\varepsilon(s,\cdot)-n_k^\varepsilon(s,\cdot)$ has a support included in $[\varepsilon,1/\varepsilon]$ for $s\in[0,t]$,
\begin{align*}
&\int_0^te^{\int_s^th(\sigma,x)\,d\sigma}\bigg(\iint_{\mathbb R_+^2}\tilde{\mathbb K}_B^\varepsilon[x_1,x_2](x)\left|n_{k+1}^\varepsilon(s,x_1)-n_{k}^\varepsilon(s,x_1)\right|\frac{n_{k+1}^\varepsilon(s,x_2)}{\int_0^{+\infty}n_{k+1}^\varepsilon(s,y)\,dy}\,dx_1\,dx_2\bigg)\,ds\\
&\quad \leq \int_0^te^{\int_s^th(\sigma,x)\,d\sigma}\bigg(\frac 1\varepsilon\|\Gamma_\varepsilon\|_{ L^\infty(\R_+)}\|n_{k+1}^\varepsilon(s,\cdot)-n_{k}^\varepsilon(s,\cdot)\|_{L^\infty(\mathbb R_+)}\int_0^{+\infty}\frac{n_{k+1}^\varepsilon(s,x_2)}{\int_0^{+\infty}n_{k+1}^\varepsilon(s,y)\,dy}\,dx_2\bigg)\,ds\\
&\quad \leq  C\int_0^{t} e^{\bar C(t-s)}\|n_{k+1}^\varepsilon(s,\cdot)-n_{k}^\varepsilon(s,\cdot)\|_{L^\infty(\mathbb R_+)}\,ds.
\end{align*}
This argument can be repeated for the last term in \eqref{eq:contraction}, while the second term on the right hand side of \eqref{eq:contraction} can be estimated thanks to the boundedness of $\tilde {\mathbb K}^\varepsilon_B[x_1,x_2]$ and to $\left|\int_0^{+\infty}n_{k}^\varepsilon(s,y)-n_{k+1}^\varepsilon(s,y))\,dy\right|\leq \frac 1\varepsilon\left\|n_{k}^\varepsilon(s,\cdot)-n_{k+1}^\varepsilon(s,\cdot))\right\|_{L^\infty(\R_+)}$. Then, with a constant $C>0$ independent of $k\in\N$, \eqref{eq:contraction} implies
\begin{align*}
\|n_{k+2}^\varepsilon(t,\cdot)-n_{k+1}^\varepsilon(t,\cdot)\|_{L^\infty(\mathbb R_+)}&\leq C\int_0^{t} e^{C(t-s)}\|n_{k+1}^\varepsilon(s,\cdot)-n_{k}^\varepsilon(s,\cdot)\|_{L^\infty(\mathbb R_+)}\,ds\\
&\leq C e^{Ct} t\sup_{s\in [0,t]}\|n_{k+1}^\varepsilon(s,\cdot)-n_{k}^\varepsilon(s,\cdot)\|_{L^\infty(\mathbb R_+)}.
\end{align*}
If we select $T>0$ small enough, $C e^{Ct} t<1$ for $t\in[0,T]$. We can then use the Banach contraction theorem to show that the sequence $(n_k^\varepsilon)_k$ converges to a limit $n^\varepsilon$. More precisely, $(n_k^\varepsilon)_k$ converges to $n^\varepsilon\in L^\infty([0,T],L^\infty(\R_+))$ for the topology of $L^\infty([0,T]\times \R_+)$. Passing to the limit $k\to\infty$ in \eqref{eq:scheme} shows that $n^\varepsilon$ satisfies \eqref{eq:truncatedpb} almost everywhere for $(t,x)\in [0,T]\times\R_+$, which implies $n^\varepsilon\in W^{1,\infty}([0,T],L^\infty(\R_+))$. The function $n^\varepsilon(t,\cdot)$ has a support included in $[\varepsilon,1/\varepsilon]$ for $t\in[0,T]$ since $n_k^\varepsilon(t,\cdot)$ has this property for $k\in\N$; similarly $n_k^\varepsilon$ is non-negative and satisfies \eqref{eq:estL1}, \eqref{est:Linf}. 

\smallskip

This construction can be iterated to extend the solution for times $t\in[0,T]$ to a solution global in time. Notice that the initial condition for this iteration (that is $n^\varepsilon(lT,\cdot)$ for $l\in\mathbb N$) does not then need to be regularized, since $n^\varepsilon(T,\cdot)$ is bounded thanks to \eqref{est:Linf}. We can then build a global solution to the truncated problem. The uniqueness of that solution results from the local contraction argument \eqref{eq:contraction}.

 \smallskip

\noindent\textbf{Step 2: Proof of estimates \eqref{eq:ineqmom2} and \eqref{eq:Lipeps}}

We have, for $\varphi\in L^1_{loc}(\R_+)$ and $t\geq 0$,
\begin{align}
&\int_0^{+\infty}\varphi(x)n^\varepsilon(t,x)\,dx =\int_0^{+\infty}\varphi(x)e^{\int_0^t h(s,x)\,ds}\left(\Gamma_{\varepsilon}\ast n^{ini}\right)(x)1_{x\leq 1/\varepsilon}\,dx\nonumber\\
&\qquad + \int_0^t \int_0^{+\infty}\varphi(x)e^{\int_s^t h(\sigma,x)\,d\sigma}\left(\Gamma_{\varepsilon}\ast g(s,\cdot)\right)(x)1_{x\leq 1/\varepsilon}\,dx\,ds
\nonumber\\
&\qquad +\int_0^t\bigg[\iint_{\mathbb R_+^2} \bigg(\int_0^{+\infty}\varphi(x)e^{\int_s^t h(\sigma,x)\,d\sigma}\tilde {\mathbb K}_B^\varepsilon[x_1,x_2](x)\,dx\bigg)\frac{n^\varepsilon(s,x_1)n^\varepsilon(s,x_2)}{\int_0^{+\infty}n^\varepsilon(s,y)\,dy}\,dx_1\,dx_2\bigg]\,ds.\label{eq:truncated}
\end{align}
We may select $\varphi(x)=x^2$, and with $\bar C>0$ denoting a constant that is independent from $\varepsilon$, we have
\begin{align*}
&\int_0^{+\infty} x^2n^\varepsilon(t,x)\,dx=\bar Ce^{\bar Ct}+\bar C\int_0^t e^{\bar C(t-s)}\,ds\\
&\qquad +\int_0^t e^{\bar C(t-s)}\bigg[\int_{\mathbb R_+^2} \bigg(\int_0^{+\infty}x^2\,d\left(\tilde{\mathbb K}_B^\varepsilon[x_1,x_2]\right)(x)\bigg)\frac{n^\varepsilon(s,x_1)n^\varepsilon(s,x_2)}{\int_0^{+\infty}n^\varepsilon(t,y)\,dy}\,dx_2\,dx_1\,ds\\
&\quad \leq \bar Ce^{\bar Ct}+\int_0^t e^{\bar C(t-s)}\bigg[\int_{\mathbb R_+^2} \bar C\big(1+x_1^2+x_2^2\big)\frac{n^\varepsilon(s,x_1)n^\varepsilon(s,x_2)}{\int_0^{+\infty}n^\varepsilon(t,y)\,dy}\,dx_2\,dx_1\bigg]\,ds\\
&\quad \leq \bar Ce^{\bar Ct}+\bar C\int_0^t e^{\bar C(t-s)}\bigg(\int_0^{+\infty} x^2n^\varepsilon(s,x)\,dx\bigg)\,ds,
\end{align*}
since we have the following inequality, that results from Remark~\ref{rem:momentsKB}:
\begin{align*}
&\int_0^{+\infty} x^2\tilde {\mathbb K}_B^\varepsilon[x_1,x_2](x)\,dx\leq \int_0^{+\infty} x^2\left(\Gamma_\varepsilon\ast \mathbb K_B[x_1,x_2]\right)(x)\,dx\\
&\quad \leq \int_0^{+\infty} (x+1)^2\,d\left( \mathbb K_B[x_1,x_2]\right)(x)\leq 2\int_0^{+\infty} 1+x^2\,d\left( \mathbb K_B[x_1,x_2]\right)(x)\leq \bar C\left(1+x_1^2+x_2^2\right).
\end{align*}
We use a Duhamel inequality  to prove \eqref{eq:ineqmom2}. With $\varphi(x)\equiv 1$, we prove the lower bound stated in \eqref{eq:ineqmom2}, thanks to the lower bound $h(x)\geq -\bar h(x)$ (with $\bar h\in L^\infty_{loc}(\R_+)$) provided by \eqref{eq:condg}:
\begin{align}
&\int_0^{+\infty}n^\varepsilon(t,x)\,dx\geq \int_0^{R}e^{\int_0^t h(s,x)\,ds}\left(\Gamma_{\varepsilon}\ast n^{ini}\right)(x)\,dx\geq \frac 1{\bar C}e^{-\bar C t},
\end{align}
where $R>0$ is chosen such that $\int_0^{R-1} \,dn^{ini}(x)>0$, and if $\varepsilon>0$ is small enough, we have $R\leq 1/\varepsilon$.  If $0<t_1\leq t_2$ and $\psi\in L^\infty(\R_+)$, then \eqref{eq:truncated} implies
\begin{align}
&\left|\int_0^{+\infty}\psi(x)\left(n^\varepsilon(t_1,x)-n^\varepsilon(t_2,x)\right)\,dx\right|\nonumber\\
&\quad \leq \left(e^{\int_0^{t_1} h(s,x)\,ds}-e^{\int_0^{t_2} h(s,x)\,ds}\right)\int_0^{+\infty}|\psi(x)|\left(\Gamma_{\varepsilon}\ast n^{ini}\right)(x)1_{x\leq 1/\varepsilon}\,dx\nonumber\\
&\qquad + \int_{0}^{t_1} \int_0^{+\infty}|\psi(x)|\,\left|e^{\int_s^{t_1} h(\sigma,x)\,d\sigma}-e^{\int_s^{t_2} h(\sigma,x)\,d\sigma}\right|\left(\Gamma_{\varepsilon}\ast g(s,\cdot)\right)(x)1_{x\leq 1/\varepsilon}\,dx\,ds\nonumber\\
&\qquad + \int_{t_1}^{t_2} \int_0^{+\infty}|\psi(x)|e^{\int_s^{t_2} h(\sigma,x)\,d\sigma}\left(\Gamma_{\varepsilon}\ast g(s,\cdot)\right)(x)1_{x\leq 1/\varepsilon}\,dx\,ds
\nonumber\\
&\qquad +\int_0^{t_1}\left|e^{\int_s^{t_1} h(\sigma,x)\,d\sigma}-e^{\int_s^{t_2} h(\sigma,x)\,d\sigma}\right|\bigg(\iint_{\mathbb R_+^2} \bigg(\int_0^{+\infty}|\psi(x)|\tilde {\mathbb K}_B^\varepsilon[x_1,x_2](x)\,dx\bigg) \frac{n^\varepsilon(s,x_1)n^\varepsilon(s,x_2)}{\int_0^{+\infty}n^\varepsilon(s,y)\,dy}\bigg)\,ds\nonumber\\
&\qquad +\int_{t_1}^{t_2}e^{\int_s^{t_2} h(\sigma,x)\,d\sigma}\bigg(\iint_{\mathbb R_+^2} \left(\int_0^{+\infty}|\psi(x)|\tilde {\mathbb K}_B^\varepsilon[x_1,x_2](x)\,dx\right)\frac{n^\varepsilon(s,x_1)n^\varepsilon(s,x_2)}{\int_0^{+\infty}n^\varepsilon(s,y)\,dy}\,dx_1\,dx_2\bigg)\,ds\nonumber\\
&\quad \leq \bar C \|\psi\|_{L^\infty(\R_+)}\Big(e^{\bar C t_1}\left(1-e^{\int_{t_1}^{t_2} h(s,x)\,ds}\right)+(t_2-t_1)e^{\bar Ct_2}+\int_0^{t_1}e^{\bar C (t_1-s)}\left|1-e^{\int_{t_1}^{t_2} h(\sigma,x)\,d\sigma}\right|\bar Ce^{\bar Cs}\,ds\Big)\nonumber\\
&\quad\leq \bar Ce^{\bar C\max(t_1,t_2)}|t_1-t_2| \,\|\psi\|_{L^\infty(\R_+)}.\label{est:calcTV}
\end{align}
thanks to the bound \eqref{eq:estL1} on $\int_0^{+\infty}n_{k}^\varepsilon(s,x)\,dx$, with $\bar C>0$ a constant independent from $\varepsilon>0$. Since this holds for any $\psi\in L^\infty(\R_+)$, we have proven (see \eqref{def:TV} for the definition of $d_{TV}$):
\[d_{TV}\Big(n^\varepsilon(t_1,\cdot)\,dx,n^\varepsilon(t_2,\cdot)\,dx\Big)\leq \bar Ce^{\bar C\max(t_1,t_2)}|t_1-t_2|.\]
\end{proof}

\subsection{Existence and uniqueness of a solution for the time-dependent problem}\label{subsec:existence}

In the following proposition, we prove the existence of measure-valued solutions for \eqref{eq:pbt}. The idea is to use a tightness and identification strategy: we prove that the approximated solutions of \eqref{eq:truncatedpb} belong to a compact set, so that a subsequence of these approximate solutions converges to a limit. We then proves that this limit is a solution of \eqref{eq:pbt}. This strategy is related to the Cauchy-Peano theorem for the existence of solutions to ordinary differential equations. Tightness-identification strategies are used in probability theory \cite{Bansaye} to construct the limit of stochastic processes that could not be done with completeness arguments. We believe this type of argument will be useful to construct solutions when a spatial structure is considered: We hope then that parabolic estimates can extend the compactness estimates used here.

\smallskip

It would probably be possible to use a more classical completeness argument (similar to the Cauchy-Lipschitz existence theorem for differential equations), using Wasserstein estimates on the transfer operator: see \cite{Magal-Raoul} for Wasserstein estimates on the transfer operator, and e.g. \cite{Dull} for the construction of measure-valued solutions using the bounded Lipschitz distance, that can be seen as generalizations of $W_1-$Wasserstein estimates to measures beyond probability measures.

\begin{proposition}\label{Prop:existence}
Let $B\in \mathcal P([0,1])$,  $g\in L^\infty(\R_+, \mathbb M_2(\R_+))$, $h\in C^0(\R_+^2,\R)$ satisfying \eqref{eq:condg}, $n^{ini}\in\mathbb M_2(\R_+)$, $\mathbb K_B$ defined as in \eqref{def:barT}. There is a unique  non-negative mild solution $n\in L^\infty_{loc}(\R_+,\mathbb M_2(\R_+))$ of \eqref{eq:pbt}, in the sense that
\begin{align}
\int_0^{+\infty}\varphi(x)\,dn_t(x)&=\int_0^{+\infty}e^{\int_0^th(s,x)\,ds}\varphi(x)\,dn^{ini}(x)+\int_0^t\int_0^{+\infty}e^{\int_s^th(\sigma,x)\,d\sigma}\varphi(x)\,dg_s(x)\,ds\nonumber\\
&\quad +\int_0^t\int_0^{+\infty}\frac{e^{\int_s^th(\sigma,x)\,d\sigma}}{\int_0^{+\infty}\,dn_s(y)}\varphi(x)\,d\left(\mathbb T_B[n_s,n_s]\right)(x)\,ds,\label{est:equalvarphi}
\end{align}
for any $\varphi\in C^0(\mathbb R_+)$ such that $|\varphi(x)|\leq C(1+x^2)$ for some $C>0$ and $x\in\R_+$. $t\mapsto n_t$ is locally Lipschitz continuous in time for the total variation distance: there is $\bar C=\bar C(n^{ini},\bar C_A,\bar h_A)>0$ depending on $n^{ini}$ and the constant appearing in \eqref{def:barT}, such that
\begin{equation}\label{eq:W1Lip}
\forall t_1,t_2\in\R_+,\quad d_{TV}\left(n_{t_1},n_{t_2}\right)\leq 2\bar Ce^{\bar C\max(t_1,t_2)}|t_1-t_2|,
\end{equation}
Moreover, $n_t$ is non-negative for $t\geq 0$, and
\begin{equation}\label{eq:ntx2}
\forall t\geq 0,\quad \int_0^{+\infty}x^2\,dn_t(x)\leq \bar Ce^{\bar Ct}.
\end{equation}

\end{proposition}
\begin{remark}
In the statement above, we state that \eqref{est:equalvarphi} holds for any continuous test function $\varphi$ satisfying $|\varphi(x)|\leq C(1+x^2)$ for $x\geq 0$. Such test functions are important to perform $W_2$ Wasserstein estimates on solutions: in \cite{Magal-Raoul}, $W_2$ Wasserstein estimates are used to prove  that the transfer operator $\mathbb T_B$ is a contraction for $W_2$ on probability measures belonging to $\mathcal P(\R_+)\cap\mathbb M_2(\R_+)$ with a given center of mass, which provides an interesting way to analyze the effect of transfer operators.
\end{remark}

\begin{remark}
If $B\in\mathcal P([0,1])$ satisfies $B\geq \alpha (\delta_0+\delta_1)$ for some $\alpha>0$ and $u\in \mathcal P(\R_+)\cap \mathbb M_2(\R_+)$, then the proof of Proposition~\ref{prop:delta01} for $\varphi\in C_c(\mathbb R_+)$ can be repeated to show:
\begin{align*}
&\int_0^{+\infty} \varphi(x){\mathbb T}_B[u,u](x)\,dx\nonumber\\
&\quad \geq \alpha^2\int_0^{+\infty}\int_0^{+\infty} \left(\int_0^1\left(\int_0^1 \varphi(x_1(1-z_1)+x_2z_2)\,d(\delta_1)(z_1)\right)\,d(\delta_0)(z_2)\right)\,d u(x_1)\,d u(x_2)\\
&\quad = \alpha^2\int_0^{+\infty}\int_0^{+\infty}\varphi(0)\,d u(x_1)\,d u(x_2)=\alpha^2\varphi(0),
\end{align*}
and then $\bar{\mathbb T}_B[u,u]\geq \alpha^2\delta_0$. This shows that if $B\geq \alpha (\delta_0+\delta_1)$, the solution $n_t$ provided by \eqref{est:equalvarphi} contains a Dirac mass at $x=0$ for any positive time, even if  $n^{ini}$ is regular (for instance $n^{ini}=\tilde n^{ini}\,dx$ with $\tilde n^{ini}\in C^\infty_c((0,\infty))$). The emergence of a Dirac mass in finite time for solutions of \eqref{eq:pbt} shows the necessity of a framework for measure-valued solutions of transfer models. Other examples of singular solutions are provided by the fixed distributions of the transfer operator described in Section~\ref{sec:fixed}, which can be seen as steady solutions of \eqref{eq:pbt} with $g_t\equiv 0$ and $h(t,x)=-1$ for $(t,x)\in\R_+^2$.

\end{remark}

\begin{proof}[Proof of Proposition~\ref{Prop:existence}]

\noindent\textbf{Step 1: Tightness - construction of a non-negative Radon measure $n_t$ for $t\in [0,T]$}

For $\varepsilon>0$, let $n^\varepsilon$ the solution of the truncated problem \eqref{eq:truncatedpb} provided by Lemma~\ref{lem:existencetruncated}. Let $T\in \mathbb N$ and for $R\in \N$, let $(\varphi_i^R)_i$ a dense subset of $ \{\varphi\in C_c(\R_+);\,\textrm{supp }\varphi\subset [0,R]\}$. For $R,i\in\mathbb N$ and $\varepsilon>0$, 
\[\lambda_\varepsilon^T[\varphi_i^R]: (t\in[0,T])\mapsto \int_0^{+\infty} \varphi_i^R(x)n^{\varepsilon}(t,x)\,dx\in\R\]
is Lipschitz continuous thanks to \eqref{eq:Lipeps}, with a Lipschitz constant $\bar Ce^{\bar CT}\|\varphi_i^R\|_{L^\infty(\R_+)}$ that is uniform in $\varepsilon>0$. Thanks to Ascoli Theorem, up to an extraction $(\varepsilon_k)_k$, the sequence $(\lambda^T_{\varepsilon_k}[\varphi_i^R])_k$ converges when $k\to\infty$ to a limit $\Lambda^T[\varphi_i^R]\in W^{1,\infty}([0,T])$:
\begin{equation}\label{def:lambda}
\left\|\lambda^T_{\varepsilon_k}[\varphi_i^R]-\Lambda^T[\varphi_i^R]\right\|_{L^\infty([0,T])}\xrightarrow[k\to\infty]{}0.
\end{equation}
We can use a diagonal extraction to construct a subsequence, that we still denote by $(\varepsilon_k)$, such that the convergence \eqref{def:lambda} holds for any $i,R\in\N$ and any final time $T\in \N$. Thanks to \eqref{def:lambda}, we have, for $i,i',R,R',T\in\N$,
\begin{align}
&\left\|\Lambda^T[\varphi_i^R]-\Lambda^T[\varphi_{i'}^{R'}]\right\|_{L^\infty([0,T])}=\lim_{k\to\infty}\left\|\lambda_{\varepsilon_k}^T[\varphi_i^R]-\lambda^T_{\varepsilon_k}[\varphi_{i'}^{R'}]\right\|_{L^\infty([0,T])}\nonumber\\
&\quad =\lim_{k\to\infty}\left\|\int_0^{+\infty} \left(\varphi_i^R(x)-\varphi_{i'}^{R'}(x)\right)n^{\varepsilon}(\cdot,x)\,dx\right\|_{L^\infty([0,T])}\leq \bar Ce^{\bar C T} \left\|\varphi_i^R-\varphi_{i'}^{R'}\right\|_{L^\infty(\R_+)},\label{eq:LambdaLip}
\end{align}
thanks to the inequality \eqref{eq:ineqmom2}, where the constant $\bar C>0$ is independent from $\varepsilon>0$. 

\smallskip

Since $(\varphi_i^R)_{i,R}$ is a dense subset of $ \big(C_c(\R_+),\|\cdot\|_{L^\infty(\R_+)}\big)$ and $\Lambda^T$ is Lipschitz continuous over this set thanks to \eqref{eq:LambdaLip}, there is a unique Lipschitz continuous extension $\bar \Lambda^T$ of $\Lambda^T$ to $ \big(C_c(\R_+),\|\cdot\|_{L^\infty(\R_+)}\big)$. If $T,T'\in \N$ with $T<T'$, then $\lambda^T_{\varepsilon_k}[\varphi_i^R](t)=\lambda_{\varepsilon_k}^{T'}[\varphi_i^R](t)$ for $t\in[0,T]$ and $i\in\N$, which implies $\Lambda^T[\varphi_i^R](t)=\Lambda^{T'}[\varphi_i^R](t)$, and then $\bar \Lambda^T[\varphi](t)=\bar \Lambda^{T'}[\varphi](t)$ for $t\in[0,T]$ and $\varphi\in C_c(\R_+)$. the function $\bar \Lambda^{T'}[\varphi]$ is then an extension of $\bar \Lambda^{T}[\varphi]$ to $[0,T']$, which allows us to define $\bar \Lambda$ as the application that satisfies $\bar \Lambda[\varphi](t)=\bar \Lambda^{T}[\varphi](t)$ for $\varphi\in C_c(\R_+)$, $t\in\R_+$ and $T>t$. Inequality \eqref{eq:Lipeps} implies that for $T\geq 0$, $\bar \Lambda[\varphi_i^R]$ is uniformly bounded in $W^{1,\infty}_{loc}([0,T])$ with a Lipschitz constant $\bar Ce^{\bar CT}\|\varphi_i^R\|_{L^\infty(\R_+)}$, which implies that for $\varphi\in  C_c(\R_+)$, $\bar \Lambda[\varphi](\cdot)$ is Lipschitz continuous with a Lipschitz constant $\bar Ce^{\bar CT}\|\varphi\|_{L^\infty(\R_+)}$ on $[0,T]$. In particular, for $\varphi\in  C_c(\R_+)$, $\bar \Lambda[\varphi](t)$ is well defined for $t\geq 0$. Since $\varphi\in C_c(\R_+)\mapsto \lambda_\varepsilon^T[\varphi](t)$ is a linear operator, the limit $\varphi\in C_c(\R_+)\mapsto \bar \lambda[\varphi](t)$ is also linear, and it is continuous since \eqref{eq:LambdaLip} implies
\[\left|\bar \Lambda[\varphi](t)-\bar \Lambda[\tilde \varphi](t)\right|\leq \|\varphi-\tilde\varphi\|_{L^\infty(\R_+)}.\]
We may apply therefore apply the Riesz–Markov–Kakutani representation theorem to define $n_t\in \mathbb M(\R_+)$ as the unique Borel measure that satisfies
\[\forall \varphi\in C_c(\R_+),\quad \bar \Lambda[\varphi](t)=\int_0^{+\infty}\varphi(x)\,dn_t(x).\]

\smallskip

\noindent\textbf{Step 2: Properties of $n_t$ for a given $t\in[0,T]$}

Let $\varphi\in C_c(\R_+)$ and $\eta>0$. Since $\varphi$ is compactly supported, there is $R\in\N$ such that $\textrm{supp }\varphi\subset [0,R]$. Thanks to the density of $(\varphi_i^R)_i$ in $\{\varphi\in C_c(\R_+);\,\textrm{supp }\varphi\subset [0,R]\}$, there is $i\in\N$ such that 
\[\|\varphi_{i}^R-\varphi\|_{L^\infty(\R_+)}\leq\eta .\]
We use this inequality and the fact that the support of $\varphi_i^R$ is included in $[0,R]$ to obtain, for $k\in\N$,
\begin{align*}
&\left|\int_0^{+\infty}\varphi(x) n^{\varepsilon_k}(t,x)\,dx-\int_0^{+\infty}\varphi(x) \,dn_t(x)\right|=\left|\int_0^{R}\varphi(x) n^{\varepsilon_k}(t,x)\,dx-\int_0^{R}\varphi(x) \,dn_t(x)\right|\\
&\quad \leq \left|\int_0^{R}\varphi_i^R(x) n^{\varepsilon_k}(t,x)\,dx-\int_0^{R}\varphi_i^R(x) \,dn_t(x)\right|+\|\varphi_i^R-\varphi\|_{L^\infty(\R_+)}\left(\int_0^R\,d|n_t|(x)+\int_0^Rn^{\varepsilon_k}(t,x)\,dx\right)\\
&\quad \leq \Big|\lambda_{\varepsilon_k}[\varphi_i^R](t)-\bar \Lambda[\varphi_i^R](t)\Big|+\eta \left(\int_0^R\,d|n_t|(x)+\int_0^Rn^{\varepsilon_k}(t,x)\,dx\right)\\
&\quad  =\left|\lambda_{\varepsilon_k}[\varphi_i^R](t)-\Lambda^{\lceil t\rceil}[\varphi_i^R](t)\right|+\eta \left(\int_0^R\,d|n_t|(x)+\bar C e^{\bar C t}\right),
\end{align*}
thanks to the estimate \eqref{eq:ineqmom2} on $\int_0^Rn^\varepsilon(t,x)\,dx$. Then, \eqref{def:lambda} implies, for $k\in\N$ large enough,
\begin{align*}
&\left|\int_0^{+\infty}\varphi(x) n^{\varepsilon_k}(t,x)\,dx-\int_0^{+\infty}\varphi(x) \,dn_t(x)\right|\leq \eta \left(1+\int_0^R\,d|n_t|(x)+\bar C e^{\bar C t}\right),
\end{align*}
and since this holds for any $\eta>0$, we have proven that,  for any $\varphi\in C_c(\R_+)$,
\begin{equation}\label{eq:cvnt}
\int_0^{+\infty}\varphi(x) n^{\varepsilon_k}(t,x)\,dx\xrightarrow[k\to\infty]{}\int_0^{+\infty}\varphi(x) \,dn_t(x).
\end{equation}
Since $n^{\varepsilon_k}$ is non-negative, for any non-negative test function  $\varphi\in C_c(\R_+)$, \eqref{eq:cvnt} implies 
\[\int_0^{+\infty}\varphi(x) \,dn_t(x)\geq 0.\]
This property, in turn, implies that $n_t$ is non-negative thanks to Proposition~\ref{prop:appendix}.

\medskip

Let $(\varphi_l)_l\in C_c(\R_+)^\N$ a non-decreasing sequence such that $\varphi_l(x)\xrightarrow[l\to\infty]{} 1+x^2$ for $x\in\R_+$. For $k,l\in \N$, the inequality \eqref{eq:ineqmom2} on $\int_0^{+\infty}\left(1+x^2\right)n^{\varepsilon_k}(t,x)\,dx$ implies $\int_0^{+\infty} \varphi_l(x)n^{\varepsilon_k}(t,x)\,dx\leq \bar Ce^{\bar Ct}$, and then, thanks to the convergence \eqref{eq:cvnt}, for $l\in\N$, $\int_0^{+\infty} \varphi_l(x)\,dn_t(x)\,dx\leq \bar Ce^{\bar Ct}$. Since $n_t$ is non-negative and $(\varphi_l)_l$ is non-decreasing, we may apply the monotone convergence theorem to show
\begin{align}\label{eq:integnt}
\int_0^{+\infty}\left(1+x^2\right)\,dn_t(x)\leq \bar Ce^{\bar C t}.
\end{align}

Let $\psi\in W^{1,\infty}(\R_+)$ and $\chi\in C(\R_+)$ such that $\chi(x)=1$ for $x\leq 1$ and $\chi(x)=0$ for $x\geq 2$, we estimate: 
\begin{align*}
&\int_0^{+\infty} \psi(x)\,d\left(n^{\varepsilon_k}(t,\cdot)\,dx-n_t\right)(x)=\int_0^{+\infty} \psi(x)\left[\chi(x/R)+\big(1-\chi(x/R)\big)\right]\,d\left(n^{\varepsilon_k}(t,\cdot)\,dx-n_t\right)(x)\\
&\leq \int_0^{+\infty} \psi(x)\chi(x/R)\,d\left(n^{\varepsilon_k}(t,\cdot)\,dx-n_t\right)(x)+\int_R^{+\infty} \left(\psi(0)+\|\psi'\|_{L^\infty(\R_+)}x\right)\,d\left(n^{\varepsilon_k}(t,\cdot)\,dx-n_t\right)(x)\\
&\quad \leq \int_0^{+\infty} \psi(x)\chi(x/R)\,d\left(n^{\varepsilon_k}(t,\cdot)\,dx-n_t\right)(x)+\frac {\bar Ce^{\bar Ct}}{R}\left(\psi(0)+\|\psi'\|_{L^\infty(\R_+)}\right),
\end{align*}
thanks to a Chebyshev's inequality and the estimates \eqref{eq:ineqmom2} (resp. \eqref{eq:integnt}) on the first and second moments of $n^{\varepsilon_k}(t,\cdot)$ (resp. $n_t$). The convergence \eqref{eq:cvnt} then implies 
\[\limsup_{k\to\infty} \int_0^{+\infty} \psi(x)\,d\left(n^{\varepsilon_k}(t,\cdot)\,dx-n_t\right)(x)\leq \frac {\bar Ce^{\bar Ct}}{R}\left(\psi(0)+\|\psi'\|_{L^\infty(\R_+)}\right),\]
and since this holds for any $R>0$, we have proven that for any $\psi\in W^{1,\infty}(\R_+)$, the following convergence holds:
\begin{align}
&\int_0^{+\infty} \psi(x)\,d\left(n^{\varepsilon_k}(t,\cdot)\,dx-n_t\right)(x)\xrightarrow[k\to\infty]{} 0. \label{cvnmW1}
\end{align}
This estimate and \eqref{eq:integnt} implies in particular $n_t\in\mathbb M_2(\R_+)$ and $\left(n^{\varepsilon_k}(t,\cdot)\,dx\right)\xrightarrow[k\to\infty]{}n_t$ for the weak topology of measures.

\smallskip

\noindent\textbf{Step 3: Identification - $n_t$ satisfies \eqref{est:equalvarphi} for $\varphi\in W^{1,\infty}(\R_+)$ with compact support}

Let $\varphi\in W^{1,\infty}(\mathbb R_+)$ with compact support. Thanks to \eqref{eq:truncatedpb}, we have, for any $t\geq 0$ and $k\in\mathbb N$,
\begin{align}
&\int_0^{+\infty}\varphi(x)\,dn_t(x)-\int_0^{+\infty}e^{\int_0^t h(s,x)\,ds}\varphi(x)\,dn^{ini}(x)-\int_0^t\int_0^{+\infty} e^{\int_s^th(\sigma,x)\,d\sigma}\varphi(x)\,dg_s(x)\,ds\nonumber\\
&\qquad -\int_0^t\iint_{\mathbb R_+^2} \left(\int_0^{+\infty}e^{\int_s^th(\sigma,x)\,d\sigma}\varphi(x)\,d\left({\mathbb K}_B[x_1,x_2]\right)(x)\right)\,\frac{dn_s(x_1)\,dn_s(x_2)}{\int_0^{+\infty}\,dn_s(y)}\,ds\nonumber\\
&\quad =\int_0^{+\infty}\varphi(x)\,d\left(n_t-n^{\varepsilon_k}(t,\cdot)\,dx\right)(x)+ \int_0^{+\infty}e^{\int_0^t h(s,x)\,ds}\varphi(x)\,d\left(n^{ini}-\left(\Gamma_{{\varepsilon_k}}\ast n^{ini}\right)1_{ \cdot\leq 1/\varepsilon_k}\,dx\right)(x)\nonumber\\
&\qquad +\int_0^t\int_0^{+\infty} e^{\int_s^th(\sigma,x)\,d\sigma}\varphi(x)\,d\left(g_s-\left((\Gamma_{\varepsilon_k}\ast g_s)1_{ \cdot\leq 1/\varepsilon_k}\right)\,dx \right)(x)\,ds\nonumber\\
&\qquad + \int_0^t\iiint_{\mathbb R_+^3} e^{\int_s^th(\sigma,x)\,d\sigma}\varphi(x)\nonumber\\
&\phantom{\qquad + }\,d\left({\mathbb K}_B[x_1,x_2]\otimes n_s\otimes n_s-(\tilde {\mathbb K}_B^{\varepsilon_k}[x_1,x_2]\,dx)\otimes(n^{\varepsilon_k}(s,\cdot)\,dx)\otimes (n^{\varepsilon_k}(s,\cdot)\,dx)\right)(x,x_1,x_2)\frac{ds}{\int_0^{+\infty}\,dn_s(y)}\nonumber\\
&\qquad + \int_0^t\iint_{\mathbb R_+^2} \left(\int_0^{+\infty}e^{\int_s^th(\sigma,x)\,d\sigma}\varphi(x)\tilde {\mathbb K}_B^{\varepsilon_k}[x_1,x_2](x)\,dx\right)\nonumber\\
&\qquad \phantom{+\int_0^trzqer}\left(\int_0^{+\infty}\,dn_s(y)-\int_0^{+\infty}n^{\varepsilon_k}(s,y)\,dy\right)\frac{dn_s(x_1)\,dn_s(x_2)}{\left(\int_0^{+\infty}\,n^{\varepsilon_k}(s,y)\,dy\right)\left(\int_0^{+\infty}\,dn_s(y)\right)}\,ds\nonumber\\
&\quad =:I_1^{\varepsilon_k}+I_2^{\varepsilon_k}+\int_0^tI_3^{\varepsilon_k}(s)\,ds+\int_0^tI_4^{\varepsilon_k}(s)\,\frac{ds}{\int_0^{+\infty}\,dn_s(y)}+\int_0^tI_5^{\varepsilon_k}(s)\,ds
\label{eq:estI1I4}
\end{align}
The term $I_1^{\varepsilon_k}$ converges to $0$ thanks to \eqref{cvnmW1}. The term  $I_2^{\varepsilon_k}$  converges to $0$ thanks to \eqref{eq:ninimnini}, since $x\mapsto e^{\int_0^t h(s,x)\,ds}\varphi(x)$ is a continuous function dominated by $x\mapsto \bar C x$. To estimate $\int_0^tI_3^{\varepsilon_k}(s)\,ds$, we notice that the limit \eqref{eq:cvg} implies $I_3^{\varepsilon_k}(s)\xrightarrow[k\to\infty ]{}0$ for $s\in[0,T]$, which we can combine to the domination:
\begin{align*}
\left|I_3^{\varepsilon_k}(s)\right|&\leq \bar C e^{\bar C(t-s)}\left\|\varphi\right\|_{L^\infty(\R_+)}\left\|s\mapsto \int_0^{+\infty} \,d g_s(x)\right\|_{L^\infty([0,T])},
\end{align*}
to apply the dominated convergence theorem and obtain $\int_0^t I_3^{\varepsilon_k}(s)\,ds\xrightarrow[k\to\infty]{} 0$.

\smallskip

To estimate $\int_0^t\int_0^{+\infty}I_4^{\varepsilon_k}(s)\,\frac{ds}{\int_0^{+\infty}\,dn_s(y)}$, we introduce the cut-off function  $\chi\in W^{1,\infty}(\R_+,[0,1])$ such that $\chi(x)=1$ for $x\leq 1$ and $\chi(x)=0$ for $x\geq 2$. We use the following decomposition of $I_4^{\varepsilon_k}(s)$:
 \begin{align*}
I_4^{\varepsilon_k}(s)&= \iiint_{\R_+^3}e^{\int_s^th(\sigma,x)\,d\sigma}\varphi(x)\chi\left(|(x_1,x_2)\|_{\R^2}/\rho\right)\\
&\phantom{\iint_{\R_+^2}}\,d\left({\mathbb K}_B[x_1,x_2]\otimes n_s\otimes n_s-(\tilde {\mathbb K}_B^{\varepsilon_k}[x_1,x_2]\,dx)\otimes(n^{\varepsilon_k}(s,\cdot)\,dx)\otimes (n^{\varepsilon_k}(s,\cdot)\,dx)\right)(x,x_1,x_2)\\
&\quad +\iiint_{\R_+^3}e^{\int_s^th(\sigma,x)\,d\sigma}\varphi(x)\left(1-\chi\left(\|(x_1,x_2)\|_{\R^2}/\rho\right)\right)\\
&\phantom{\iint_{\R_+^2}}\,d\left({\mathbb K}_B[x_1,x_2]\otimes n_s\otimes n_s-(\tilde {\mathbb K}_B^{\varepsilon_k}[x_1,x_2]\,dx)\otimes(n^{\varepsilon_k}(s,\cdot)\,dx)\otimes (n^{\varepsilon_k}(s,\cdot)\,dx)\right)(x,x_1,x_2)\\
&=:I_{4,1}^{\varepsilon_k}(s)+I_{4,2}^{\varepsilon_k}(s),
 \end{align*}
 with $\rho>0$. Then, $(x_1,x_2)\mapsto \chi\left(\|(x_1,x_2)\|_{\R^2}/\rho\right)$ is supported in $\{(x_1,x_2)\in \R_+^2; \,\|(x_1,x_2)\|_{\R^2}\leq \rho\}$, and  for $(x_1,x_2)$ in this set, the measure $\mathbb K_B[x_1,x_2]$ is supported in $[0,2\rho]$.  It implies that $(x,x_1,x_2)\mapsto e^{\int_s^th(\sigma,x)\,d\sigma}\varphi(x)\chi\left(|(x_1,x_2)\|_{\R^2}/\rho\right)$ is a compactly supported continuous function. Thanks to \eqref{eq:cvnt}, for $s\geq 0$, $n^{\varepsilon_k}(s,\cdot)\,dx$ converges to $n_s$ for the weak topology of measures over $\R_+$ when $k\to\infty$. Thanks to its definition \eqref{def:Kkeps}, $\tilde {\mathbb K}_B^{\varepsilon_k}[x_1,x_2]\,dx$ converges to $\mathbb K_B[x_1,x_2]$ for the weak topology of measures over $\R_+$ when $k\to\infty$. Theorem 2.8 in \cite{Billingsley} then implies the convergence of $(\tilde {\mathbb K}_B^{\varepsilon_k}[x_1,x_2]\,dx)\otimes(n^{\varepsilon_k}(s,\cdot)\,dx)\otimes (n^{\varepsilon_k}(s,\cdot)\,dx)$ to ${\mathbb K}_B[x_1,x_2]\otimes n_s\otimes n_s$ for the weak topology of measures over $\R_+^3$ when $k\to\infty$. Then, for $s\in[0,t]$,
\[I_{4,1}^{\varepsilon_k}(s)\xrightarrow[k\to\infty]{}0.\]
We combine this limit result to the following domination:
\begin{align*}
|I_{4,1}^{\varepsilon_k}(s)|&\leq \left\|(x,x_1,x_2)\mapsto e^{\int_s^th(\sigma,x)\,d\sigma}\varphi(x)\chi\left(|(x_1,x_2)\|_{\R^2}/\rho\right)\right\|_{L^\infty(\R_+)}\\
&\quad \iint_{\R_+^2} \,d\left({\mathbb K}_B[x_1,x_2]\otimes n_s\otimes n_s+(\tilde {\mathbb K}_B^{\varepsilon_k}[x_1,x_2]\,dx)\otimes(n^{\varepsilon_k}(s,\cdot)\,dx)\otimes (n^{\varepsilon_k}(s,\cdot)\,dx)\right)(x,x_1,x_2)\\
&\leq \bar C e^{\bar Ct}\|\varphi\|_{L^\infty(\R_+)},
\end{align*}
and the dominated convergence theorem implies
\begin{equation}\label{eq:cvI41}
\int_0^tI_{4,1}^{\varepsilon_k}(s)\,\frac{ds}{\int_0^{+\infty}\,dn_s(y)}\xrightarrow[k\to\infty]{}0.
\end{equation}

To estimate $\int_0^tI_{4,2}^{\varepsilon_k}(s)\,\frac{ds}{\int_0^{+\infty}\,dn_s(y)}$, we notice that 
 \begin{align}
 \left|\frac d{dx}\left(e^{\int_s^t h(\sigma,\cdot)\,d\sigma}\varphi\right)(x)\right|&=\left|e^{\int_s^t h(\sigma,x)\,d\sigma}\varphi'(x)+e^{\int_s^t h(\sigma,x)\,d\sigma}\varphi(x)\int_s^t \partial_x h(\sigma,x)\,d\sigma\right|\nonumber\\
 &\leq \bar C e^{\bar C t}\|\varphi'\|_{L^\infty(\R_+)}+\bar Ce^{\bar Ct}\left(|\varphi(0)|+x\|\varphi'\|_{L^\infty(\R_+)}\right) t\|\partial_x h\|_{L^\infty(\R_+^2)}\nonumber\\
 &\leq \bar C_\varphi e^{\bar Ct} (1+t)(1+x),\label{eq:Lipes}
 \end{align}
 where $\bar C_\varphi>0$ designates a constant that may depend on $\varphi\in W^{1,\infty}(\R_+)$. This estimate and Remark~\ref{rem:momentsKB} imply: 
\begin{align*}
I_{4,2}^{\varepsilon_k}(s)&\leq \iint_{\R_+^2}\left(\int_0^{+\infty}\bar C_\varphi e^{\bar Ct} (1+t)(1+x)\,d(\tilde{\mathbb K}^{\varepsilon_k}_B[x_1,x_2]\,dx+\mathbb K_B[x_1,x_2])(x)\right)\left(1-\chi\left(\|(x_1,x_2)\|_{\R^2}/\rho\right)\right)\\
&\phantom{\leq \iint_{\R_+^2}}d\left(n_s\otimes n_s+(n^{\varepsilon_k}(s,\cdot)\,dx)\otimes (n^{\varepsilon_k}(s,\cdot)\,dx)\right)(x_1,x_2)\\
&\leq \bar C_\varphi e^{\bar Ct}\iint_{\R_+^2}(1+x_1+x_2)(1_{x_1\geq \rho}+1_{x_2\geq \rho})\,d\left(n_s\otimes n_s+(n^{\varepsilon_k}(s,\cdot)\,dx)\otimes (n^{\varepsilon_k}(s,\cdot)\,dx)\right)(x_1,x_2)\\
&\leq \frac{\bar C_\varphi}{\rho} e^{\bar Ct}\iint_{\R_+^2}(1+x_1^2+x_2^2)\,d\left(n_s\otimes n_s+(n^{\varepsilon_k}(s,\cdot)\,dx)\otimes (n^{\varepsilon_k}(s,\cdot)\,dx)\right)(x_1,x_2)\\
&\leq \frac{\bar C_\varphi e^{\bar C t}}{\rho},
\end{align*}
thanks to a Chebyshev's inequality and estimates \eqref{eq:ineqmom2}, \eqref{eq:integnt} on the second moments of $n_s$, $n^{\varepsilon_k}(s,\cdot)$. Then,
\[\limsup_{k\to\infty}\int_0^tI_{4,2}^{\varepsilon_k}(s)\,\frac{ds}{\int_0^{+\infty}\,dn_s(y)}\leq \frac{\bar C_\varphi e^{\bar C t}}{\rho}.\]
Combining this estimate, that holds for any $\rho>0$, and estimate \eqref{eq:cvI41}, we obtain
\begin{equation}\label{eq:cvI43}
\int_0^tI_{4}^{\varepsilon_k}(s)\,\frac{ds}{\int_0^{+\infty}\,dn_s(y)}\xrightarrow[k\to\infty]{}0.
\end{equation}

\smallskip

To estimate $I_5^{\varepsilon_k}$, we observe that \eqref{cvnmW1} implies
\[\left(\int_0^{+\infty}n^{\varepsilon_k}(s,y)\,dy-\int_0^{+\infty}\,dn_s(y)\right)\xrightarrow[k\to\infty]{}0,\]
while $\left|\int_0^{+\infty}e^{\int_s^th(\sigma,x)\,d\sigma}\varphi(x)\tilde{\mathbb K}^{\varepsilon_k}_B[x_1,x_2](x)\,dx\right|\leq\bar Ce^{\bar C(t-s)}\left\|\varphi\right\|_{L^\infty(\R_+)}$, that provides the following domination of $I_5^{\varepsilon_k}(s)$: 
\begin{equation*}
\left|I_5^{\varepsilon_k}(s)\right|\leq  \bar C e^{\bar C(t-s)}\left\|\varphi\right\|_{L^\infty(\R_+)},
\end{equation*}
thanks to \eqref{eq:ineqmom2} and \eqref{eq:integnt}. The dominated convergence theorem therefore implies $\int_0^tI_5^{\varepsilon_k}(s)\,ds\xrightarrow[k\to\infty]{}0$. 

\smallskip

Thanks to the estimates we have made on the terms on the right hand side of \eqref{eq:estI1I4}, a limit $k\to\infty$ proves \eqref{est:equalvarphi} for $\varphi\in W^{1,\infty}(\R_+)$ with a compact support.

\smallskip

\noindent\textbf{Step 4: Proof that $n_t$ satisfies \eqref{est:equalvarphi} for general $\varphi$ and uniqueness of mild solutions}

Let now $\varphi\in C^0(\mathbb R_+)$ satisfying $|\varphi(x)|\leq C(1+x^2)$, and $t\geq 0$. We define the following sequence of functions:
\[\varphi_l(x)=(\psi\ast \Gamma_{1/l})(x)\chi(x/l),\]
where $\Gamma_{1/l}$ is defined by \eqref{def:Kkeps} and $\chi\in W^{1,\infty}(\R_+)$ satisfies $\chi(x)=1$ for $x\leq 1$ and $\chi(x)=0$ for $x\geq 2$. Then, $\varphi_l\in W^{1,\infty}(\mathbb R_+)$ for $l\in\mathbb N$ and we have $\varphi_l(x)\xrightarrow[l\to\infty]{}\varphi(x)$ for $x\in\R_+$. \eqref{est:equalvarphi} is then satisfied for $\varphi:=\varphi_l$, and we now show that we can pass to the limit $l\to\infty$ in each term of the equation to prove that \eqref{est:equalvarphi} holds with $\varphi:=\psi$. To do so, we use the dominated convergence Theorem. We notice that we have the uniform bound $|\varphi_l(x)|\leq C(1+x^2)$, and $x\mapsto C(1+x^2)$ belongs to the set of integrable functions on $\R_+$ for the measure $n_t$ (thanks to \eqref{eq:integnt}), as well as for the measure $e^{\int_0^t h(\sigma,\cdot)\,d\sigma}n^{ini}$, and then, thanks to the dominated convergence theorem,
\[\int_0^{+\infty}\varphi_l(x)\,dn_t(x)\xrightarrow[l\to\infty]{}\int_0^{+\infty}\varphi(x)\,dn_t(x),\]
\[ \int_0^{+\infty}e^{\int_0^th(s,x)\,ds}\varphi_l(x)\,dn^{ini}(x)\xrightarrow[l\to\infty]{}\int_0^{+\infty}e^{\int_0^th(s,x)\,ds}\varphi(x)\,dn^{ini}(x).\]
We can also notice that $(t,x)\mapsto C(1+x^2)$ belongs to the set of integrable functions on $[0,t]\times \R_+$ for the measure $e^{\int_s^th(\sigma,\cdot)\,d\sigma}\varphi(\cdot)\,dg_s\,dt$, and then
\[\int_0^t\int_0^{+\infty}e^{\int_s^th(\sigma,x)\,d\sigma}\varphi_l(x)\,dg_s(x)\,ds\xrightarrow[l\to\infty]{}\int_0^t\int_0^{+\infty}e^{\int_s^th(\sigma,x)\,d\sigma}\varphi(x)\,dg_s(x)\,ds\]
And $(t,x)\mapsto C(1+x^2)$ also belongs to the set of integrable functions on $[0,t]\times \R_+$ for the measure $ \frac{e^{\int_s^th(\sigma,\cdot)\,d\sigma}}{\int_0^{+\infty}\,dn_s(y)}\mathbb T_B[n_s,n_s]\,ds$, thanks to \eqref{eq:integnt} and \eqref{est:moments2}. Then,
\[\int_0^t\int_0^{+\infty}\frac{e^{\int_s^th(\sigma,x)\,d\sigma}}{\int_0^{+\infty}\,dn_s(y)}\varphi_l(x)\,d\left(\mathbb T_B[n_s,n_s]\right)(x)\,ds\xrightarrow[l\to\infty]{}\int_0^t\int_0^{+\infty}\frac{e^{\int_s^th(\sigma,x)\,d\sigma}}{\int_0^{+\infty}\,dn_s(y)}\varphi(x)\,d\left(\mathbb T_B[n_s,n_s]\right)(x)\,ds,\]
and the equality \eqref{est:equalvarphi} then holds for any $\varphi\in C^0(\R_+)$ such that $\varphi(x)|\leq C(1+x^2)$ for $x\in\R_+$.

\medskip

To prove the uniqueness of mild solutions, we reproduce the calculation \eqref{est:calcTV}: for $\psi\in L^\infty(\R_+)\cap C_c(\mathbb R_+)$ with $\|\psi\|_{L^\infty(\R_+)}\leq 1$, if $n$, $m$ satisfying \eqref{est:equalvarphi} for any $\varphi\in C_c(\R_+)$ and $\int_0^t\,dn_t\leq \bar Ce^{\bar Ct}$, $\int_0^t\,dm_t\leq \bar Ce^{\bar Ct}$ for some $\bar C>0$ and $t\geq 0$, then
\begin{align}
&\int_0^{+\infty}\psi(x)\,dn_t(x)-\int_0^{+\infty}\psi(x)\,dm_t(x)\nonumber\\
&\quad \leq \int_0^{t}e^{\int_s^th(\sigma,x)\,dx}\bigg[\iint_{\mathbb R_+^2} \left(\int_0^{+\infty}\psi(x)\,d{\mathbb K}_B[x_1,x_2](x)\right)\frac{dn_s(x_1)\,dn_s(x_2)}{\int_0^{+\infty}\,dn_s(y) }\bigg]\,ds\nonumber\\
&\qquad - \int_0^{t}e^{\int_s^th(\sigma,x)\,dx}\bigg[\iint_{\mathbb R_+^2} \left(\int_0^{+\infty}\psi(x)\,d{\mathbb K}_B[x_1,x_2](x)\right)\frac{dm_s(x_1)\,dm_s(x_2)}{\int_0^{+\infty}\,dm_s(y) }\bigg]\,ds\nonumber\\
&\quad \leq \int_0^{t}e^{\int_s^th(\sigma,x)\,dx}\Bigg\{\iint_{\mathbb R_+^2} \left(\int_0^{+\infty}\psi(x)\,d{\mathbb K}_B[x_1,x_2](x)\right)\frac{d(n_s-m_s)(x_1)\,dn_s(x_2)}{\int_0^{+\infty}\,dn_s(y) }\nonumber\\
&\qquad +\iint_{\mathbb R_+^2} \left(\int_0^{+\infty}\psi(x)\,d{\mathbb K}_B[x_1,x_2](x)\right)\left(\int_0^{+\infty}\,d(m_s-n_s)(y)\right)\frac{dn_s(x_1)}{\int_0^{+\infty}\,dn_s(y) }\frac{dm_s(x_2)}{\int_0^{+\infty}\,dm_s(y) }\bigg]\,ds\nonumber\\
&\qquad +\iint_{\mathbb R_+^2} \left(\int_0^{+\infty}\psi(x)\,d{\mathbb K}_B[x_1,x_2](x)\right)\frac{dm_s(x_1)\,d(n_s-m_s)(x_2)}{\int_0^{+\infty}\,dm_s(y) }\Bigg\}\,ds\nonumber\\
&\quad \leq 3\int_0^t\left(e^{\bar C(t-s)}\sup_{\|\varphi\|_{L^\infty(\R_+)}\leq 1}\int_0^{+\infty} \varphi(y)\,d(n_s-m_s)(y)\right)\,ds.\label{est:uniq}
\end{align}
Then, denoting
\begin{align*}
Y(s):= \sup_{\|\varphi\|_{L^\infty(\R_+)}\leq 1}\int_0^{+\infty}\varphi(x)\,d\left(n_s-m_s\right)(x),
\end{align*}
we can consider the supremum of \eqref{est:uniq} over $\psi\in C_c(\R_+)$ with $\|\psi\|_{L^\infty(\R_+)}\leq 1$. Thanks to Lusin’s theorem (see Section 2.24 of [32]), this supremum is equal to $Y(t)$, and then
\[Y(t)\leq \int_0^{t} Ce^{Cs} e^{s-t} Y(s)\,ds.\]
A Gronwall estimate then shows $Y\equiv 0$, that is $n=m$, hence the uniqueness of mild solutions.
\end{proof}

\section*{Appendix}

\begin{proposition}\label{prop:appendix}
Let $\nu$ a Borel measure on $\R_+$. Assume that for any non-negative function $\varphi\in C_c(\R_+)$, we have
\[\int_0^{+\infty}\varphi(x)\,d\nu(x)\geq 0,\]
then $\nu$ is non-negative, that is $\nu(A)\geq 0$ for any Borel set $A$.
\end{proposition}

\begin{proof}[Proof of Proposition~\ref{prop:appendix}]
Assume first that $A$ is a bounded set: $A\subset [0,R]$ for some $R\geq 0$. Let $\varphi:=1_A\in L^\infty(\R_+)$. We use Lusin's theorem (see  Section 2.24 of \cite{Rudin}): Since $\nu|_{[0,R+1)}$ is a Borel measure on $[0,R+1)$, there is a sequence $(\varphi_l)_l\in (C_c([0,R+1)))^\N$ with $0\leq \varphi_l\leq 1$ (a minor modification of the corollary from Section 2.24 of \cite{Rudin} is necessary: in that reference, $\varphi_l$ satisfies $\|\varphi_l\|_{L^\infty}\leq \|\varphi\|_{L\infty}$, for this adaptation, one should consider $\tilde \varphi:=\varphi-1/2$) such that $\varphi(x)=\lim_{l\to\infty}\varphi_l(x)$ almost everywhere for $\nu|_{[0,R+1)}$, and then, for $l\in\N$,
\begin{align*}
\int_0^{+\infty}\varphi(x) \,d\nu(x)&\geq -\left|\int_0^{+\infty}\left(\varphi(x)-\varphi_l(x)\right) \,d\nu(x)\right|+\int_0^{+\infty}\varphi_l(x) \,d\nu(x)\\
&\geq -\left|\int_0^{R}\left(\varphi(x)-\varphi_l(x)\right) \,d\nu(x)\right|,
\end{align*}
Thanks to the definition of $\varphi_l$, we have $\varphi(x)=\lim_{l\to\infty}\nu(x)$ almost everywhere for $\nu|_{[0,R+1)}$, and we have the domination $\left|\varphi(x)-\varphi_l(x)\right|\leq 2$ with $\int_0^R2\,d|\nu|(x)<\infty$, so that the dominated convergence theorem can be used to show $\int_0^{R}\left(\varphi(x)-\varphi_l(x)\right) \,d\nu(x)\to 0$ as $l\to \infty$. This proves $\nu(A)=\int_0^{+\infty}\varphi(x) \,d\nu(x)\geq 0$. 
\end{proof}

In the next lemma, we derive some estimates on the initial condition, source term and regularized transfer kernels of the truncated model \eqref{eq:truncatedpb} that are used in Section~\ref{subsec:existence}:
\begin{lemma}\label{lemma:approx}
Let $B\in \mathcal P([0,1])$,  $g\in L^\infty(\R_+, \mathbb M_2(\R_+))$, $h\in C^0(\R_+^2,\R)$ satisfying \eqref{eq:condg}, $n^{ini}\in\mathbb M_2(\R_+)$,  $\mathbb K_B$ defined as in \eqref{def:barT}. Let $\psi\in C(\R_+,\R)$ satisfying $|\psi(x)|\leq \bar C x$ for a constant $\bar C=\bar C(n^{ini},B,\bar C_A,\bar h_A)>0$ and any $x\in\R_+$. Then, for $t\geq 0$,
\begin{equation}\label{eq:ninimnini}
\int_0^{+\infty} \psi(x) \,d\left(\left(\left(\Gamma_{\varepsilon}\ast n^{ini}\right)1_{ \cdot\leq 1/\varepsilon}\right)\,dx-n^{ini}\right)(x)\xrightarrow[\varepsilon\to 0]{}0.
\end{equation}
\begin{align}
&\int_0^{+\infty} \psi(x) \,d\bigg(\left((\Gamma_\varepsilon\ast g_t)1_{ \cdot\leq 1/\varepsilon}\right)\,dx-g_t\bigg)(x)\xrightarrow[\varepsilon\to 0]{}0.\label{eq:cvg}
\end{align}
\end{lemma}

\begin{proof}[Proof of Lemma~\ref{lemma:approx}]
\begin{align}
&\left|\int_0^{+\infty} \psi(x) \,d\left(\left(\left(\Gamma_{\varepsilon}\ast n^{ini}\right)1_{ \cdot\leq 1/\varepsilon}\right)\,dx-n^{ini}\right)(x)\right|\nonumber\\
&\quad \leq\left|\int_0^{+\infty} \psi(x) \,d\left(\left(\Gamma_{\varepsilon}\ast n^{ini}\right)\,dx-n^{ini}\right)(x)\right|+ \left|\int_{1/\varepsilon}^{+\infty} \psi(x) \left(\Gamma_{\varepsilon}\ast n^{ini}\right)(x)\,dx\right|\nonumber\\
&\quad \leq \left|\int_0^{+\infty} \left(\int_0^{+\infty} \left(\psi(y)-\psi(x)\right)\Gamma_{\varepsilon}(x-y)\,dy\right) \,dn^{ini}(x)\right|+ \bar C\int_{1/\varepsilon}^{+\infty} (1+x^2)\,d n^{ini}(x).\label{eq:est:lemnini}
\end{align}
Since $n^{ini}\in \mathbb M_2(\R_+)$,  we have $\int_0^{+\infty}(1+x^2)\,dn^{ini}(x)<\infty$, which implies $\int_{1/\varepsilon}^{+\infty} (1+x^2)\,d n^{ini}(x)\xrightarrow[]{\varepsilon\to 0}0$. Moreover,
\[\left|\int_0^{+\infty} \left(\psi(y)-\psi(x)\right)\Gamma_{\varepsilon}(x-y)\,dy\right|\leq 2 \|\psi'\|_{L^\infty(\R_+)}\int_0^{+\infty} y\Gamma_\varepsilon(y)\,dy, \]
and the right hand side of that inequality defines a function in $L^1(n^{ini})$. This domination and the convergence $\int_0^{+\infty} \left(\psi(y)-\psi(x)\right)\Gamma_{\varepsilon}(x-y)\,dy\xrightarrow[\varepsilon\to 0]{}0$ allow us to use the dominated convergence theorem to show
\[\int_0^{+\infty} \left(\int_0^{+\infty} \left(\psi(y)-\psi(x)\right)\Gamma_{\varepsilon}(x-y)\,dy\right) \,dn^{ini}(x)\xrightarrow[\varepsilon\to 0]{}0,\]
and \eqref{eq:ninimnini} is proven. This argument can be repeated to show \eqref{eq:cvg}.

\end{proof}

\section*{Acknowledgements}
This work has been supported by the Chair Modélisation Mathématique et Biodiversité of Veolia - Ecole
polytechnique - Museum national d’Histoire naturelle - Fondation X. It is also funded by the European Union
(ERC AdG SINGER, 101054787) and the ANR project  DEEV (ANR-20-CE40-0011-01). Pierre Magal has sadly passed away before the completion of this manuscript.

\end{document}